\documentclass[12pt]{amsart}

\usepackage{url}
\usepackage{graphicx}
\usepackage{amsmath}
\usepackage{amscd}
\usepackage{amsfonts}
\usepackage{amssymb}
\usepackage{color}
\usepackage{fullpage}
\usepackage{mathtools}
\usepackage[pagebackref,hypertexnames=false, colorlinks, citecolor=red,linkcolor=blue, urlcolor=red]{hyperref}
\usepackage{todonotes}
\usepackage{enumitem}
\usepackage{comment}
\usepackage{cleveref}

\numberwithin{equation}{section}

\newtheorem{theorem}{Theorem}[section]
\newtheorem{lemma}[theorem]{Lemma}
\newtheorem{proposition}[theorem]{Proposition}

\newtheorem{quest}[theorem]{Question}

\newtheorem{corollary}[theorem]{Corollary}

\newtheorem{thmA}{Theorem}

\theoremstyle{definition}
\newtheorem{example}[theorem]{Example}
\newtheorem{remark}[theorem]{Remark}
\newtheorem{definition}[theorem]{Definition}

\newcommand{\be}{\begin{equation}}
\newcommand{\ee}{\end{equation}}
\newcommand{\bes}{\begin{equation*}}
\newcommand{\ees}{\end{equation*}}

\newcommand{\cA}{\mathcal{A}}

\newcommand{\cE}{\mathcal{E}}
\newcommand{\cF}{\mathcal{F}}

\newcommand{\cH}{\mathcal{H}}

\newcommand{\cK}{\mathcal{K}}

\newcommand{\cM}{\mathcal{M}}
\newcommand{\cN}{\mathcal{N}}
\newcommand{\cO}{\mathcal{O}}

\newcommand{\bB}{\mathbb{B}}
\newcommand{\bC}{\mathbb{C}}
\newcommand{\C}{\mathbb{C}}
\newcommand{\bD}{\mathbb{D}}

\newcommand{\bM}{\mathbb{M}}
\newcommand{\bN}{\mathbb{N}}

\newcommand{\ol}{\overline}

\newcommand{\Tr}{\operatorname{Tr}}

\newcommand{\alg}{\operatorname{alg}}

\newcommand{\id}{\operatorname{id}}

\newcommand{\fB}{{\mathfrak{B}}}

\newcommand{\fD}{{\mathfrak{D}}}

\newcommand{\dom}{\operatorname{Dom}}
\newcommand{\GL}{\mathrm{GL}}
\DeclareMathOperator{\oa}{OA_u}
\DeclareMathOperator{\cmax}{C^*_{\max}}
\DeclareMathOperator{\ran}{ran}

\DeclareMathOperator{\op}{op}
\DeclareMathOperator{\Ad}{Ad}

\begin{document}

\title{On the spectral radius of operator tuples}

\author{Marcel Scherer}
 \address{M.S., Faculty of Mathematics\\
 Technion Israel Institute of Technology\\
 Technion City, Haifa\; 3200003\\
 Israel}

 \author{Orr Moshe Shalit}
 \address{O.M.S., Faculty of Mathematics 
 \\The Helen Diller Quantum Center
 \\Technion Israel Institute of Technology\\
 Technion City, Haifa\; 3200003\\
 Israel}
 
 \urladdr{\href{https://oshalit.net.technion.ac.il/}{\url{https://oshalit.net.technion.ac.il/}}}

 \author{Eli Shamovich}
 \address{E.S., Department of Mathematics\\
 Ben-Gurion University of the Negev\\
 Beer-Seva\; 8410501\\
 Israel}

 \thanks{The work of E. Shamovich is partially supported by BSF grant no.\ 2022235.}
 

\addcontentsline{toc}{section}{Abstract}

\begin{abstract}
In recent work, Shalit and Shamovich associated to every operator space structure $\cE$ on $\bC^d$ a spectral radius function $\rho_\cE$ on $d$-tuples of operators. 
The main goal of this paper is to elucidate how this spectral radius depends on the operator space structure. 
Let $V = (\bC^d, \|\cdot\|_V)$ be a normed space and let $\cE$ be a quantization of $V$. 
We show that for a commuting operator tuple $X$, the spectral radius depends only on the underlying normed space; more precisely, 
\[
\rho_\cE(X) = \max\{ \|\lambda\|_V : \lambda \in \sigma(X)\}, 
\]
where $\sigma(X)$ denotes the joint spectrum of $X$.
In contrast, we prove that if $\dim V \geq 3$, then $\rho_{\min(V)}(X) \neq \rho_{\max(V)}(X)$ already for some matrix tuple $X$. 
When $\cE_1$ and $\cE_2$ are selfadjoint operator spaces, we show that $\rho_{\cE_1}(X) = \rho_{\cE_2}(X)$ for all tuples $X$ implies $\cE_1 = \cE_2$. 
We present two proofs of this result; a key ingredient in one of them is a characterization, of independent interest, of $\rho_\cE(A)$ in terms of the invertibility domain of the linear pencil associated with $A$. 
Finally, we prove that if two operator spaces give rise to the same spectral radius function, then the algebras of locally uniformly bounded NC functions on the corresponding NC unit balls coincide.  
\end{abstract}

\maketitle

\section{Introduction}

\subsection{Background and motivation}

The spectrum and, by extension, the spectral radius are classical invariants of matrices, operators, and more generally elements of Banach algebras. Let $\rho(a) = \sup\{|\lambda| : \lambda \in \sigma(a)\}$ denote the spectral radius of an element $a$ in a Banach algebra $A$. The classical Gelfand formula states that
\[
\rho(a) = \lim_{n \to \infty} \|a^n\|^{1/n}.
\]
The spectral radius as a function on a Banach algebra captures essential properties of the algebra. For example, a result of Aupetit \cite[Corollary 5.2.3]{Aupetit-primer} states that $A/\operatorname{Rad}(A)$ is commutative if and only if $\rho$ is uniformly continuous (alternatively, $\rho$ is subadditive or submultiplicative). When $A$ is a $C^*$-algebra, it was proved by Shulman that the spectral radius is continuous if and only if $A$ is of Type I \cite{Shulman-spec_rad_cont}. Linear maps $A \to B$ preserving the spectrum of elements were studied by many following the works of Kadison and others (see for example \cite{Aupetit-spec_preserv_map, AupMou-spec_preserv_map} and the references therein). As for the spectral radius, it was shown by Bre\v{s}ar and \v{S}merl that if $A = B(X)$, where $X$ is a Banach space, then a linear map $\varphi \colon A \to A$ preserves the spectral radius of every element if and only if $\varphi = c \psi$, where $c$ is a unimodular constant and $\psi \colon A \to A$ is given by $\psi(T) = S T S^{-1}$, where $S \colon X \to X$ is invertible or $X$ is reflexive and $\psi(T) = S T^t S^{-1}$, where $S \colon X^* \to X$ is a bounded bijection.

Different notions of joint spectrum for a commuting tuple of operators or elements of a Banach algebra and corresponding spectral radii were defined by many authors (see \cite{Curto-applications} for an excellent overview). In the noncommutative setting the Gelfand formula led Bunce \cite{Bunce} and later Popescu \cite{Popescu-similarity} to define the joint spectral radius of a tuple $X_1,\ldots,X_d \in B(\cH)$ by
\[
\rho_{\operatorname{row}}(X) = \lim_{n \to \infty} \left\| \sum_{|\alpha| = n} X^{\alpha} (X^{\alpha})^* \right\|^{1/2n}.
\]
This joint spectral radius was successfully applied to various problems in noncommutative analysis \cite{arora2024optimal, augat2025operator, JMS-rat, JMS-clark, Pascoe21, Pop14}. In \cite{shalsham2025spectral}, Shalit and Shamovich introduced the notion of a joint spectral radius of $d$-tuples of operators with respect to an operator space structure $\cE$ on $\C^d$ generalizing the spectral radius of Bunce and Popescu; this new notion has already found applications in NC function theory \cite{sampat2026cyclicity}. In fact, two spectral radii corresponding to an operator space $\cE$ were proposed in \cite{shalsham2025spectral}. 
A naive definition gives rise to the {\em minimal spectral radius} $\rho^{\min}_\cE$. 
An adaptation of this idea gives rise to the ``correct" spectral radius $\rho_{\cE}$ (we recall the necessary definitions below). The main result in \cite{shalsham2025spectral} is that an operator tuple $X$ is jointly similar to a point in the appropriate unit ball $\bB_\cE^{(\infty)}$ if and only if $\rho_\cE(X)<1$. 
We reprove this in Theorem \ref{thm:SS25} below, extending the result also for infinite dimensional operator spaces. The goal of this paper is to study the dependence of the spectral radius $\rho_\cE$ on the operator space structure $\cE$. The first main result of the paper is 
\begin{thmA}[Theorems \ref{thm:comm_specrad} and \ref{thm:min_specrad}]
    Let $V$ be $\C^d$ equipped with the norm $\|\cdot\|_V$. Let $X \in B(\cH)^d$ be a commuting $d$-tuple. Then, for every operator space structure $\cE$ over $V$ the spectral radii $\rho_{\cE}(X)$ and $\rho^{\min}_{\cE}(X)$ coincide and are given by 
    \[
    \rho_{\cE}(X) = \rho^{\min}_{\cE}(X) = \max\{ \|\lambda\|_V \mid \lambda \in \sigma(X) \}.
    \]
    Here, $\sigma(X)$ is the Banach algebra joint spectrum of the tuple $X$. In particular, the joint spectral radius of a commuting tuple is independent of the quantization.
\end{thmA}

On the other hand, one would expect that in general the spectral radii of two different quantizations of the norm structure on $\C^d$ should differ. However, care is needed with the verb ``differ'': it is not hard to check that if $\rho_{\operatorname{col}}$ is the joint spectral radius corresponding to the column operator space structure, then for every $d$-tuple of matrices $X$ of any size, $\rho_{\operatorname{row}}(X) = \rho_{\operatorname{col}}(X)$ (see Example \ref{ex:row_col_equal} and \cite[Lemma 2.1]{JMS-clark}); but, as observed in Example \ref{ex:row_col_equal} (see also Example \ref{ex:not_persist}), there is a $d$-tuple of operators that can distinguish these spectral radii. In Example \ref{ex:row_col_2} we show that some spectral radii of quantizations can be distinguished on the  level of matrices. More generally, we prove

\begin{thmA}[Theorem \ref{thm:selfadj_ineq}]\label{thm:B}
    Let $\cE_1$ and $\cE_2$ be two different selfadjoint operator spaces quantizing the same involutive Banach space structure $V$ on $\C^d$. Then, there exists a $d$-tuple of matrices such that $\rho_{\cE_1}(X) \neq \rho_{\cE_2}(X)$.
\end{thmA}
In Theorem \ref{thm:rho_E1E2} we obtain a weaker version of Theorem \ref{thm:B}, by which the spectral radii of two different selfadjoint operator spaces can be separated by operator tuples. 
The advantage of this method is that it uses a different proof, which might generalize to the case of arbitrary operator spaces. A particular case of Theorem \ref{thm:B} is obtained if we let $V$ be $\C^d$ equipped with a norm $\|\cdot\|_V$ such that $\|z\|_V = \|\bar{z}\|_V$ for every $z \in \C^d$, then either $\rho_{\min(V)} \neq \rho_{\max(V)}$ or $V$ admits a unique quantization. The latter claim generalizes as follows.
\begin{thmA}[Theorem \ref{thm:min_neq_max}]
    Let $V$ be a complex Banach space with $\dim V=d\geq 3$. Then there exist
$n\in\mathbb N$ and $X\in M_n(V)$ such that
  \[
    \rho_{\min(V)}(X)<\rho_{\max(V)}(X).
  \]
\end{thmA}

We know by the example of the column and row spaces that equality of the spectral radii for all matrix tuples does not imply that $\cE_1 = \cE_2$. We leave the following question open. 
\begin{quest}\label{quest:rho_E_1_2}
    Let $\cE_1$ and $\cE_2$ be operator space structures on $\C^d$ such that for every $d$-tuple of operators $X$, $\rho_{\cE_1}(X) = \rho_{\cE_2}(X)$. Must the identity map be completely isometric?
\end{quest}

The noncommutative open unit ball $\bB_\cE$ of $\cE$ consists of all matrices $X \in M_m(\cE)$ such that $\|X\|_{M_n(\cE)}<1$, for all $n \in \bN$ (we write $\bB_\cE^{(\infty)}$ for the analogous ball with an additional level at $n = \infty$). 
In terms of noncommutative complex analysis, the logarithm of the joint spectral radius is analogous to a plurisubharmonic exhaustion function of the open unit ball $\bB_{\cE}$ of $\cE$. We can define  the topological algebra $\cO(\bB_{\cE})$ of all NC analytic functions of bounded type on $\bB_{\cE}$ (see Section \ref{sec:rho_E_F} for more details). 
\begin{thmA}[Theorem \ref{thm:spr_OBE}]
    Let $\cE_1$ and $\cE_2$ be two operator spaces such that for every $d$-tuple of operators $X$, $\rho_{\cE_1}(X) = \rho_{\cE_2}(X)$. Then, $\cO(\bB_{\cE_1}) = \cO(\bB_{\cE_2})$ as sets and the identity map is a complete homeomorphism.
\end{thmA}
However, the equality of the spectral radii on all {\em operator} tuples is not a necessary condition, as Example \ref{ex:insufficent_OBE} shows. Therefore, it seems natural that the topological equality $\cO(\bB_{\cE_1}) = \cO(\bB_{\cE_2})$ is equivalent to the fact that $\rho_{\cE_1}(X) = \rho_{\cE_2}(X)$ for all $d$-tuples of matrices of all sizes.

Lastly, we extend the results of \cite{JMS-rat} and \cite{shalsham2025spectral}, which connect the spectral radius of a tuple $A$ to the domain of invertibility of the corresponding linear pencil $L_A(X) = I - \sum_j X_j \otimes A_j$, to the case of operator realizations as constructed in \cite{augat2025operator}. This is a key ingredient in our proof of the operator version of Theorem \ref{thm:B} (Theorem \ref{thm:rho_E1E2}). 

\begin{thmA}[Theorem \ref{thm:rVSdom}]
    Given an operator space structure $\cE$ on $\bC^d$, let $E = \cE^*$ denote the dual operator space structure. 
    For $A \in B(\cH)^d$ and $r > 0$, the following are equivalent:
    \begin{enumerate}
        \item $\rho_E(A) < r^{-1}$. 
        \item $R\bB_\cE^{(\infty)}\subset \dom^{(\infty)}(L^{-1}_A)$ for some $R > r$. 
        \item The similarity envelope of $R\bB^{(\infty)}_\cE$ is contained in $\dom^{(\infty)}(L^{-1}_A)$ for some $R > r$.
        \item For some $R > r$ the pencil $L_A$ is invertible in $A(R\bB_\cE) \otimes B(\cH)$. 
    \end{enumerate}
    In particular, for all $X \in B(\cK)^d$, if $\rho_\cE(X) \rho_E(A) < 1$ then $X \in \dom(L_A^{-1})$.
\end{thmA}
Significantly, the operator level is necessary for this theorem to hold (see Example \ref{ex:not_persist}). In fact, the function constructed in this example is analogous to the example of a freely entire NC function that is unbounded on the unit row ball constructed by Pascoe \cite{Pascoe20}. 

\subsection{Definitions and notation}

Let $M_n$ denote the $n \times n$ matrices over $\bC$. 
For an operator space $\cE$, we let $M_n(\cE) = M_n \otimes \cE$ and define $\bM(\cE) = \cup_{n=1}^\infty M_n(\cE)$. We remind the reader that an {\em operator space} $\cE$ is just a subspace of a C*-algebra with the inherited norm. 
This implies that $M_n(\cE)$ too is an operator space.

Although our primary interest is finite dimensional operator spaces, we shall also treat general operator spaces when this doesn't complicate the arguments too much. We shall identify a $d$-dimensional operator spaces as $\bC^d$ equipped with a particular operator space structure (i.e., a family of matrix norms), in which $M_n(\bC^d) = M_n^d$ is the set of all $d$-tuples of scalar matrices, and we then write $\bM^d$ for $\bM(\bC^d) = \cup_{n=1}^\infty M_n^d$. 
A subset $\Omega \subset \bM(\cE)$ is said to be an {\em NC set} if it is closed under direct sums. 
We write $\Omega_n$ or $\Omega(n)$ for $\Omega \cap M_n(\cE)$, so $\Omega = \cup_{n=1}^\infty \Omega_n$. 
The {\em similarity envelope} $\widetilde{\Omega}$ of an NC set $\Omega$ consists of all tuples $S^{-1} X S := (S^{-1}X_1S, \ldots, S^{-1}X_d S)$ where $X \in \Omega$, that is
\begin{equation}\label{eq:sim_env}
\widetilde{\Omega} = \bigcup_{n=1}^\infty\left\{S^{-1} X S : X \in \Omega_n, \,\, S \in \GL_n \right\}.
\end{equation}
An {\em NC domain} is an open and levelwise connected NC set. 

Whenever we use topological notions without further specification, we shall mean the disjoint union topology. 
For example, we write $\ol{\Omega} := \cup_n \ol{\Omega}_n$. 
There are other topologies of interest; most relevant for us is the {\em uniform} topology given on $\bM(\cE)$ when $\cE$ is an operator space, which is the topology generated by open balls of the form 
\[
B(X, r) = \cup_{n=1}^d \{Y \in \bM(\cE) : \|Y - \oplus_{k=1}^n X\|<r\} 
\]
for $X \in \bM(\cE)$ and $r>0$. 

Every operator space $\cE$ gives rise to the {\em operator space ball} (or {\em operator ball}, for short) $\bB_\cE$
\begin{equation}\label{eq:op_ball}
\bB_\cE = \cup_{n=1}^\infty \{X \in M_n(\cE) : \|X\|_n<1\}.
\end{equation}

When $\cE = \bC^d$ equipped with a particular operator space structure given by the inclusion $e_j \mapsto Q_j$ mapping the standard basis vector $e_j \in \bC^d$ to an operator $Q_j \in B(\cH)$, then we can think of $\bB_\cE$ as a subset of $\bM^d$ as follows
\[
\bB_\cE = \left\{X \in \bM^d : \|X\|_n = \left\|\sum_{j=1}^d X_j \otimes Q_j  \right\| < 1\right\} .
\]
Following \cite{BMV-ncrkhs,BMV-int,SampatShalit25,sampat2025weak}, when considering such balls we sometimes use the notation $\bD_Q$ (a notation meant to indicate ``the unit disc determined by $Q$") instead of $\bB_\cE$, to highlight the depends on the choice of coordinates, in other words the dependence on $Q_1, \ldots, Q_d$. 
We shall write $\bD_Q^{(\infty)} = \bD_Q \sqcup \bD_Q(\infty)$, where 
\[
\bD_Q(\infty) = \left\{X \in B(\ell^2)^d : \left\|\sum_{j=1}^d X_j \otimes Q_j \right\| < 1 \right\}
\]
can be considered to be an ``operator level" added to all the matrix levels in $\bD_Q$. 
Similarly, we use the notation $\bB_\cE^{(\infty)}$ to denote the disjoint union of $\bB_\cE$ with an added operator level --- the open unit ball of $B(\ell^2) \otimes_{\min} \cE$.

It will be useful to have a couple of concrete examples at hand. 
The {\em row ball} $\fB_d$, which is defined to be $\bB_\cE$ for $\cE=\bC^d$ with the row operator space structure, is given by
\begin{equation}\label{eq:Ball}
\fB_d := \bB_{(\bC^d)_{\operatorname{row}}} = \{ X \in \bM^d : \|X\|_{\operatorname{row}} := \|\sum X_j X_j^*\|^{1/2} < 1\} .
\end{equation}
Likewise, one defines the column operator space $(\bC^d)_{\operatorname{col}}$ and the corresponding {\em column ball}. 
When $\cE = \min(\ell^\infty_d)$ is the minimal operator space over $\ell^\infty_d$ we get the {\em NC polydisc}
\[
\fD_d := \bB_{\min(\ell^\infty_d)} = \{X \in \bM^d : \|X\|_\infty := \max_{1 \leq i \leq d} \|X_i\| < 1\}. 
\]

A function $f$ from an NC set $\Omega \subset \bM(\cE)$ to $\mathbb{M}(\cF)$ is said to be an {\em NC function} if: (i) $f$ is graded: $X \in \Omega_n \Rightarrow f(X) \in M_n(\cF)$; (ii) $f$ respects direct sums: $f(X \oplus Y) = f(X) \oplus f(Y)$; and (iii) $f$ respects similarities: if $X \in \Omega_n$, $S \in GL_n$, and if $S^{-1} X S \in \Omega_n$, then $f(S^{-1} X S) = S^{-1} f(X) S$. 

We let $\bC\langle z \rangle = \bC \langle z_1, \ldots, z_d \rangle$ denote the algebra of free polynomials in $d$ noncommuting variables, also referred to as the {\em free algebra}. 
A more sophisticated class of NC functions is provided by {\em NC rational functions} \cite{helton2018applications,JMS-rat,KVV-rat,mai2018free,Volcic}. 
Remarkably, every bounded NC function is {\em analytic} and has a Taylor series at every point; in fact, for an NC function in an operator ball, the series around the origin converges in the entire ball \cite{KVV}.

We define $H^\infty(\bB_\cE)$ to be the algebra of bounded NC functions on $\bB_\cE$ equipped with the supremum norm
\[
\|f\| = \|f\|_\infty := \sup_{X \in \bB_\cE}\|f(X)\|, 
\] 
and define $A(\bB_\cE)$ to be the closure of polynomials in $H^\infty(\bB_\cE)$. 
By \cite{SampatShalit25}, $A(\bB_\cE)$ equals the set of bounded NC functions on $\bB_\cE$ that extend to uniformly continuous functions on $\ol{\bB_\cE}$. 

For an operator space $E$, we let $\oa(E)$ denote the universal unital operator algebra generated by $E$ (see \cite[Chapter 6]{Pisier-book}). 
In \cite[Section 7.1]{SampatShalit25} it was shown that for a finite dimensional operator space $E$, one has the completely isometric isomorphism $\oa(E) \cong A(\bB_{E^*})$. 
In fact, the same argument works to show this for general operator spaces (in \cite{SampatShalit25} it is also true that $\oa(\cE^*) \cong A(\bB_{\cE})$ because of finite dimensionality).

\section{Operator space structures and spectral radii}

\subsection{The joint spectral radius associated with an operator space}

In \cite{shalsham2025spectral}, for every operator space $\cE$, a spectral radius $\rho_\cE$ was defined for matrices over $\cE$, and more generally, for elements in $B(\cK) \otimes_{\min} \cE$. 
Recall that if $X \in B(\cK) \otimes_{\min} \cE$, then
\begin{equation}\label{eq:spec_rad}
    \rho_\cE(X) := \rho(\iota_\cE(X)), 
\end{equation}
where $\iota_\cE$ is shorthand for the ampliation ${\bf id}_{B(\cK)} \otimes\iota_\cE$ of the inclusion map $\iota_\cE \colon \cE \to \cmax(\cE)$. 
That is, $\rho_\cE(X)$ is defined to be the spectral radius $\rho(\iota_\cE(X))$ of the element $\iota_\cE(X)$ in the C*-algebra $B(\cK) \otimes_{\min} \cmax(\cE)$.

When considering tuples of operators it is convenient to have a concrete, coordinate dependent version of the above spectral radius. 
We will view a finite dimensional operator space $\cE$ concretely as the space $\bC^d$ endowed with a family of matrix norms. 
Every such identification determines a family of norms on $d$-tuples of operators. 
Letting $Q_1, \ldots, Q_d \in B(\cH)$ be the image of the standard basis vectors $e_1, \ldots, e_d$ under a completely isometric embedding of $\bC^d$ into $B(\cH)$, we define for $X \in B(\cK)^d$ the spectral radius as 
\[
\rho_\cE(X) = \rho_Q(X) := \rho\left(\sum_j X_j \otimes Q_j\right)
\]
where the spectral radius on the right is calculated in $B(\cK) \otimes_{\min} \cmax(\cE)$ as above. 
In \cite{shalsham2025spectral} it was shown that
\begin{equation}\label{eq:Hsr}
\rho_Q(X) = \lim_{n \to \infty} \left \|\sum_{|w|=n} X^w \otimes Q_{w_1} \otimes_h Q_{w_2} \otimes_h \cdots \otimes_h Q_{w_n} \right\|^{1/n} 
\end{equation}
where the sum is over all words $w = w_1 w_2 \cdots w_n$ in $d$ letters $w_i \in \{1,2,\ldots, d\}$.
The tensors in \eqref{eq:Hsr} are to be understood as elements in $B(\cK) \otimes \left(\cE \otimes_h \cdots \otimes_h \cE\right)$, where $\cE \otimes_h \cdots \otimes_h \cE$ is the $n$th-fold Haagerup tensor product of $\cE$ with itself. 
Note that when $d = 1$ and $\cE = \bC$ with the standard operator space structure, formula \eqref{eq:Hsr} reduces to the spectral radius formula.

For $S \in B(\cH)$ invertible, we can consider the map $\Ad_S \otimes \id \colon B(\cH) \otimes_{\min} \cE \to B(\cH) \otimes_{\min} \cE$. For $X \in B(\cH) \otimes_{\min} \cE$ we will write $S^{-1} X S = (\Ad_S \otimes \id)(X)$. Clearly, $\rho_{\cE}(X) = \rho_{\cE}(S^{-1} X S)$.

The main result about the joint spectral radius associated with an operator space is Theorem 2.7 in \cite{shalsham2025spectral}, which can be restated as follows.

\begin{theorem}[\cite{shalsham2025spectral}]\label{thm:SS25}
    Let $\cE$ be an operator space and let $X \in B(\cH) \otimes \cE$. Then $\rho_\cE(X)<1$ if and only if there exists an invertible $S \in B(\cH)$ such that $\left\| S^{-1} X S \right\| < 1$. 
\end{theorem}
In \cite{shalsham2025spectral} the above theorem was established only for the case that $\cE$ is an operator space structure on $\bC^d$. 
With some care, the proof generalizes to the case of not-necessarily-finite dimensional operator spaces, as we now demonstrate.
\begin{proof}
    One direction is clear. Now consider the canonical completely isometric embedding $\cE \hookrightarrow \cE^{**}$. This map induces a homomorphism $\iota \colon C^*_{\max}(\cE) \to C^*_{\max}(\cE^{**})$. Moreover, since every completely contractive map $\cE \to B(\cK)$ extends to a normal completely contractive map $\cE^{**} \to B(\cK)$, we have that if $f \in \ker \iota$, then by the universal property of the maximal $C^*$-algebras, we have that $f$ is annihilated by every representation of $C^*_{\max}(\cE)$ and, thus, $f =0$. Therefore, $\iota$ is injective and since these are $C^*$-algebras, we have that the spectral radius of an element is preserved. 
    The same holds true for the minimal tensor product with $B(\cH)$ by injectivity. In other words, we have proved that $\rho_{\cE}(X) = \rho_{\cE^{**}}(X)$, for all $X \in B(\cH) \otimes_{\min} \cE$. By \cite[Equation  1.35]{BlecherLeMerdy}, we have a completely isometric embedding $B(\cH) \otimes_{\min} \cE^{**} \hookrightarrow \operatorname{CB}(\cE^*, B(\cH))$. More generally, we combine this with \cite[Equation 1.43]{BlecherLeMerdy} to get that
    \[
    B(\cH) \otimes_{\min} (\cE^{**})^{\otimes_h n} \hookrightarrow B(\cH) \otimes_{\min} ((\cE^*)^{\otimes_h n})^* \hookrightarrow \operatorname{CB}((\cE^*)^{\otimes_h n}, B(\cH)).
    \]
    Here all the embeddings are completely isometric. By \cite[Proposition 6.6]{Pisier-book}, $(\cE^*)^{\otimes_h n}$ is completely isometrically isomorphic to the subspace of homogeneous elements of degree $n$ in the universal unital operator algebra $\oa(\cE^*)$. 
    Since $\rho_{\cE}(X) < 1$, it follows that $\|X^n\| \leq r^n$ for some $0 < r < 1$  and for $n$ sufficiently large. 
    Therefore, the map induced by $X$ on $\cE^*$ gives rise to a completely bounded homomorphism $\varphi_X \colon \oa(\cE^*) \to B(\cH)$. By Paulsen's similarity theorem, there is an invertible $S \in B(\cH)$, such that $S^{-1} \varphi_X S$ is completely contractive. Lastly, note that for every $g \in \cE^*$, $\varphi_X(g) = (\id_{B(\cH)} \otimes \operatorname{ev}_g)(X)$. Hence, $S^{-1} \varphi_X(g) S = \varphi_{S^{-1} X S}(g)$. Therefore, $\|S^{-1} X S\| \leq 1$. Now the positive homogeneity of the spectral radius allows us to finish the proof by replacing $X$ by $(1+\epsilon)X$ for some $\epsilon > 0$.
\end{proof}
When thinking of $X$ as a $d$-tuple of operators $X = (X_1, \ldots, X_d) \in B(\cK)^d$, the theorem says that $\rho_\cE(X)<1$ if and only if there exists an invertible operator $S \in B(\cK)$ such that the tuple $S^{-1} XS := (S^{-1} X_1 S, \ldots, S^{-1} X_d S)$ is in the unit ball of $B(\cH) \otimes \cE$, i.e., the operators $X_1, \ldots, X_d$ are simultaneously similar to a point in $\bB_\cE^{(\infty)}$.

\subsection{The spectral radius and representations}
Let $V = (\C^d, \|\cdot\|_V)$ be a Banach space structure on $\C^d$. Recall from \cite{PaulsenBook} that there are minimal and maximal operator space structures on $\C^d$ that quantize $V$. We will denote them by $\min(V)$ and $\max(V)$, respectively. By their universal properties, for every other operator space structure $\cE$ on $\C^d$ quantizing $V$, the identity map gives rise to completely contractive maps 
\[
\max(V) \to \cE \to \min(V). 
\]
In particular, we have that 
\[
\bB_{\max(V)} \subset \bB_{\cE} \subset \bB_{\min(V)}.
\]
Additionally, by the universal properties of the unital operator and maximal $C^*$-algebra of an operator system, we have homomorphisms 
\[
\oa(\max(V)) \to \oa(\cE) \to \oa(\min(V))
\]
and 
\[
\cmax(\max(V)) \to \cmax(\cE) \to \cmax(\min(V)). 
\]
The former have dense range and the latter are $*$-homomorphisms, and thus, are quotient maps. Recall that the sup norm closure of the free algebra on $\bB_{\cE}$ is canonically isomorphic to $\oa(\cE^*)$. Moreover, $\max(V)^* = \min(V^*)$ and $\min(V)^* = \max(V^*)$. In particular, $ A(\bB_{\min(V)}) \cong \oa(\max(V^*))$ and $A(\bB_{\max(V)}) \cong \oa(\min(V^*))$. We summarize this discussion with an application to the spectral radii of tuples. 
We shall write $Z = (Z_1, \ldots, Z_d)$ for the tuple of coordinate functions in $A(\bB_\cE)$. 

In the following lemma, we let $\cE$ be an operator space structure on $\bC^d$ generated by operators $Q_1, \ldots, Q_d$, and let $E = \cE^*$ denote its operator space dual. 
Then the coordinate functions $Z_1, \ldots, Z_d \in A(\bB_\cE)$ generate the operator space structure on $\bC^d$ corresponding to $E$. 
\begin{lemma}\label{lem:specrad_homo}
Let $\cE$, $E$, $Q$ and $Z$ be as above. Then, 
\be\label{eq:claim}
\rho_{E}(Q) = \rho_{\cE}(Z) = 1 = \left\|\sum_k Z_k \otimes Q_k \right\| .
\ee
Consequently, if $\varphi \colon A(\bB_\cE) \to B(\cK)^d$ is a unital completely bounded homomorphism and $X = (X_1, \ldots, X_d) = \varphi(Z) := (\varphi(Z_1), \ldots, \varphi(Z_d))$, 
    then, $\rho_\cE(X) \leq 1$.
\end{lemma}
\begin{proof}
    The closed ball $\ol{\bB_{\cE}}$ corresponds to all completely contractive representations of $E$ into finite dimensional space, thus by an elementary property of operator spaces (see \cite[Remark 2.1.2]{Pisier-book}), we have
    \[
    \left\|\sum_k Z_k \otimes Q_k \right\| = \sup\left\{ \left\| \sum_k X_k \otimes Q_k \right\| : X \in \ol{\bB_{\cE}} \right\} = \sup\left\{ \|X\|_{\cE} : X \in \ol{\bB_{\cE}} \right\} = 1. 
    \]
    Together with Theorem \ref{thm:SS25}, this implies that $\rho_{\cE}(Z) \leq 1$. 
    On the other hand, by \cite[Remark 2.4]{shalsham2025spectral} (which says that the spectral radius decreases under completely bounded homomorphisms), we have that $\rho_{\cE}(X) \leq \rho_{\cE}(Z)$ for all $X \in \bB_\cE$. 
    But by Theorem \ref{thm:SS25}, we have, for every $r<1$, a point $X \in \bB_\cE$ such that $\rho_\cE(X) > r$. It follows that $\rho_\cE(Z) = 1$. 
    The equality $\rho_E(Q) = 1$ follows by symmetry. 

By \cite[Remark 2.4]{shalsham2025spectral}, we have that the spectral radius can only decrease under the application of unital completely bounded homomorphisms, thus $\rho_\cE(\varphi(Z))\leq\rho_\cE(Z)$. 
\end{proof}

    Note that the converse of Lemma \ref{lem:specrad_homo} does not hold in general: $\rho_\cE(X) \leq 1$ does not imply that $X$ gives rise to a completely bounded representation of $A(\bB_\cE)$, as can be witnessed already in the $d = 1$ case by considering 
    \[
      X = 
      \begin{pmatrix}
      1 & 1 \\
      0 & 1
      \end{pmatrix} .
    \]
However, there are two variants of a partial converse that do hold: first, if $\cK$ is finite dimensional and $X$ is irreducible, then $\rho_\cE(X) \leq 1$ implies that $X$ gives rise to a completely bounded representation of $A(\bB_\cE)$; and second, in general, the strict inequality $\rho_\cE(X) < 1$ implies that $X$ gives rise to a bounded homomorphism. 
We record these two facts in the following two lemmas. 

\begin{lemma}\label{lem:specrad_irreducible}
        Let $\cE$ be an operator space structure on $\C^d$, let $\cK$ be finite dimensional and let $X \in B(\cK)^d$ be irreducible. If $\rho_\cE(X) \leq 1$, then $X$ gives rise to a unital completely bounded homomorphism $\varphi \colon A(\bB_\cE) \to B(\cK)^d$ such that $X = (X_1, \ldots, X_d) = \varphi(Z) := (\varphi(Z_1), \ldots, \varphi(Z_d))$. 
\end{lemma}

\begin{proof}
We think of $X$ as an element in $M_n(\cE)$. 
By \cite[Corollary 2.12]{shalsham2025spectral}, there exists $S\in GL_n$ such that $\|S^{-1} X S\|_{M_n(\cE)} \leq 1$. 
It follows that $Y = S^{-1} X S$ gives rise to a completely contractive representation and therefore $X = S Y S^{-1}$ gives rise to a completely bounded representation. 
\end{proof}

\begin{lemma}\label{lem:specrad_strict}
        Let $\cE$ be an operator space structure on $\C^d$, let $X \in B(\cK)^d$. If $\rho_\cE(X) < 1$, then $X$ gives rise to a unital completely bounded homomorphism $\varphi \colon A(\bB_\cE) \to B(\cK)^d$ such that $X = (X_1, \ldots, X_d) = \varphi(Z) := (\varphi(Z_1), \ldots, \varphi(Z_d))$. 
\end{lemma}
\begin{proof}
    Let $Q_1, \ldots, Q_d \in B(\cH)$ be a basis that endows $\C^d$ with the operator system structure $\cE$. By the main result of \cite[Theorem 2.7]{shalsham2025spectral}, $\rho_\cE(X) < 1$ implies that $X$ is similar to a point in $\bD_Q^{(\infty)}$. 
    We may therefore assume without loss of generality that $X \in \bD_Q^{(\infty)}$, that is, that 
    \begin{equation}\label{eq:TQcontraction}
    \left\| \sum_j X_j \otimes_{\min} Q_j \right\| < 1.
    \end{equation}
    Now, recall from \cite{SampatShalit25,sampat2025weak} that by the very definition of the algebra of bounded NC functions on $\bB_\cE$, every point $X \in \bD_Q$ gives rise to an evaluation representation $\Phi_X \colon f \mapsto f(X)$ defined on $H^\infty(\bB_\cE)$, and therefore also defined on $A(\bB_\cE)$. Therefore, it remains to prove for $X \in B(\cK)^d$ at the infinite level that if \eqref{eq:TQcontraction} holds, then $Z_i \mapsto X_i$ can be extended to a completely contractive mapping from $A(\bB_\cE)$ to $B(\cK)$. 

    Now, let $X \in B(\cK)^d$ such that \eqref{eq:TQcontraction} holds. Clearly, $X$ gives rise to a unital algebra homomorphism $\varphi$ from the free algebra $\bC\langle z_1, \ldots, z_d \rangle$ to $B(\cK)$ taking $z_i$ to $X_i$. 
    We want to show that $\varphi$ is completely bounded, when $\bC \langle z_1, \ldots, z_d \rangle$ inherits the norm from $A(\bB_\cE)$. 
    For every finite-dimensional subspace $\cM \subset \cK$, we consider $X^\cM = P_\cM X \big|_\cM = (P_\cM X_1 \big|_\cM, \ldots, P_\cM X_d \big|_\cM)$. 
    Then $X^\cM$ can be considered as a point in $\bD_Q$, and therefore gives rise to unital completely contractive homomorphism $\varphi^\cM \colon A(\bB_\cE) \to B(\cM) \subset B(\cK)$. For every free polynomial $p \in \bC \langle z_1, \ldots, z_d \rangle$, 
    \[
    \varphi^\cM(p) \xrightarrow{\cM \nearrow \cK} \varphi(p) 
    \]
    in the strong operator topology. 
    It follows that $\varphi$ is a unital, completely contractive homomorphism from $\bC\langle z_1, \ldots, z_d \rangle$ to $B(\cK)$, and it extends to the desired completely contractive unital homomorphism $\varphi \colon A(\bB_\cE) \to B(\cK)$ which maps $Z$ to $X$, as required. 
\end{proof}

\begin{lemma} \label{lem:specrad_ineq}
    Let $\cE$ be an operator space structure on $\C^d$ that quantizes $V$, then for every $d$-tuple $X \in B(\cK)^d$,
    \begin{equation}\label{eq:inequalities}
    \rho_{\min(V)}(X) \leq \rho_{\cE}(X) \leq \rho_{\max(V)}(X).
    \end{equation}
\end{lemma}
\begin{proof}
    This follows from the fact that the $*$-homomorphisms 
    \[
    B(\cK) \otimes_{\min} \cmax(\max(V)) \to B(\cK) \otimes_{\min} \cmax(\cE) \to B(\cK) \otimes_{\min} \cmax(\min(V))
    \]
    map $\iota_{\max(V)}(X)$ to $\iota_\cE(X)$ and then to $\iota_{\min(V)}(X)$, whence 
    \[
    \rho(\iota_{\min(V)}(X)) \leq \rho(\iota_\cE(X)) \leq \rho(\iota_{\max(V)}(X)). 
    \]
    It follows from the definition \eqref{eq:spec_rad} of the spectral radius that \eqref{eq:inequalities} holds. 

    Alternatively, this follows from Lemma \ref{lem:specrad_homo} as follows. Since the spectral radius is homogeneous, it suffices to show that 
    \[
    \rho_{\max(V)}(X) < 1 \implies \rho_{\cE}(X) \leq 1, 
    \]
    and similarly for $\rho_{\cE}(X)$ and $\rho_{\min(V)}(X)$. 
    Suppose that $\rho_{\max(V)}(X) < 1$.
    Then by Lemma \ref{lem:specrad_strict} $Z_i \mapsto X_i$ defines a completely bounded unital homomorphism $\varphi \colon A(\bB_{\max(V)}) \to B(\cK)$. 
    This gives rise to a completely bounded unital homomorphism 
    \[
    \tilde{\varphi} \colon A(\bB_\cE) \to A(\bB_{\max(V)}) \to B(\cK)
    \]
    sending $Z$ to $X$. By Lemma \ref{lem:specrad_homo}, $\rho_\cE(X) \leq 1$. 
    The analogous implication $\rho_{\cE}(X) < 1 \implies \rho_{\min(V)}(X) \leq 1$ is shown in the same way. 
\end{proof}

\section{Dependence on the operator space structure}

\subsection{The case of commuting tuples}

Let $V$ be a normed space and let $B_V$ be the open unit ball of $V$. We have a canonical isometric map $V^* \to C(\overline{B_V})$. 
This embedding induces the minimal operator space structure on $V^*$. Moreover, the algebra generated by the image of $V^*$ is the ball algebra $A(B_V)$, i.e., the algebra of all continuous functions on $\ol{B_V}$ which are analytic in $B_V$. 
By the universal properties described above, we have a completely contractive homomorphism $A(\bB_{\max(V)}) \to A(B_V)$; in fact the map is given concretely by restriction $A(\bB_\cE) \ni f \mapsto f \big|_{B_V}$ (note that this map is in general not bounded below or surjective). Therefore, if $\varphi \colon A(B_V) \to B(\cK)$ is a completely bounded homomorphism, then, by Lemma \ref{lem:specrad_homo}, $\rho_{\max(V)}(\varphi(z)) \leq 1$. Here, $\varphi(z) = (\varphi(z_1),\ldots,\varphi(z_d)) \in B(\cK)^d$ is a commuting tuple. 

For a commuting tuple $X \in B(\cK)^d$ we will let $\cA_X$ denote the unital operator algebra generated by $X$, and let $\sigma(X)$ denote the joint spectrum of $X$ with respect to $\cA_X$. Namely,
\[
\sigma(X) = \left\{ (\psi(X_1),\ldots, \psi(X_d)) \in \C^d \mid \psi \text{ a character of } \cA_X\right\}, 
\]
where by {\em character} we mean a nonzero homomorphism of $\cA_X$ into $\bC$. 
References for the Banach algebraic spectrum and functional calculus are \cite[Chapter I]{Gamelin-unifrom_alg} and \cite[Section III.4]{Gamelin-unifrom_alg}. 

\begin{lemma} \label{lem:specrad_implies_spectrum}
    Let $\cE$ be an operator space structure on $\C^d$ quantizing $V$. 
    For every commuting tuple $X = (X_1,\ldots,X_d) \in B(\cK)^d$, if $\rho_{\cE}(X) < 1$, then $\sigma(X) \subset B_V$.
\end{lemma}
\begin{proof}
Since $\rho_{\cE}(X) < 1$, we know by Lemma \ref{lem:specrad_strict} that the map $A(\bB_{\cE}) \to B(\cK)$ that sends the coordinates to $X$ is completely bounded. Then, every character of the unital algebra generated by $X$ pulls back to a character of $A(\bB_{\cE})$. Therefore, $\sigma(X) \subset \ol{\bB_\cE(1)} = \overline{B_V}$. Since this is true for $(1+\varepsilon)X$ where $\varepsilon > 0$ is sufficiently small, we see that $\sigma(X) \subset B_V$.
\end{proof}

We shall also use the notion of spectrum introduced by Taylor \cite{Taylor-joint_spectrum}, which is denoted by $\sigma_T(X)$. 
The Taylor spectrum $\sigma_T(X)$ coincides with the Banach algebraic spectrum $\sigma(X)$ when $\cK$ is finite dimensional. 
In general we have $\sigma_T(X) \subset \sigma(X)$ (see \cite[Proposition 25.3]{muller2007spectral} or \cite{Curto-applications,Taylor-joint_spectrum}). 
The following theorem is further justification for the term ``spectral radius". 

\begin{theorem} \label{thm:comm_specrad}
    Let $X = (X_1,\ldots,X_d) \in B(\cK)^d$ be a commuting tuple, let $V$ be $\bC^d$ equipped with a norm $\|\cdot\|_V$, and let $\cE$ be an operator space structure on $\C^d$ quantizing $V$. Then, $\rho_{\cE}(X) < 1$ if and only if $\sigma(X) \subset B_V$, and this happens if and only if $\sigma_T(X) \subset B_V$. 
\end{theorem}
\begin{proof}
    One implication follows form Lemma \ref{lem:specrad_implies_spectrum} and another implication follows immediately from the fact that $\sigma_T(X) \subset \sigma(X)$ noted above. 
    
    By Lemma \ref{lem:specrad_ineq}, to prove the remaining implication it suffices to prove that $\sigma_T(X) \subset B_V$ implies that $\rho_{\max(V)}(X) < 1$. Assume that $\sigma_T(X) \subset B_V$. 
    Briefly, by the Taylor functional calculus, there exists a bounded homomorphism $A(B_V) \to B(\cK)$ mapping $f$ to $f(X)$; in fact, since the functional calculus is given by an integral formula, this map is completely bounded. 
    Moreover, the functional calculus lifts to a completely bounded homomorphism $A(\bB_{\max(V)}) \to B(\cK)$, whence $\rho_{\max(V)}(X) \leq 1$, by Lemma \ref{lem:specrad_homo}. A $1+\varepsilon$ perturbation argument gives $\rho_{\max(V)}(X) < 1$, as required.
    
    Let us provide some more details, since we have not found a convenient reference for the complete boundedness of the Taylor functional calculus. 
    Since $\sigma_T(X)$ is compact we can find an open set $U$ with smooth boundary, such that $\sigma_T(X) \subset U \subset \overline{U} \subset B_V$. Therefore, by Vasilescu's Martinelli formula for the analytic functional calculus, for every $f \in A(B_V)$,
    \[
    f(X) \xi = \frac{1}{(2\pi i)^d} \int_{\partial U} f(z) (M_X(z) \xi) \wedge dz_1 \wedge \cdots \wedge d z_d.
    \]
    Here, $M_X$ is an operator-valued differential form of degree $(0,n-1)$. To define $M_X$ we need to consider the differential of the Koszul complex on $\cK$. Recall that we have the Hilbert space $\Lambda(\cK) = \oplus_{m=0}^d (\wedge^m \C^d) \otimes \cK$. Let us fix an orthonormal basis $e_1,\ldots, e_d$ for $\C^d$ and write $\omega$ for an element of $\Lambda(\cK)$ and an operator on it
    \[
    \delta_X (\xi \otimes \omega) = \sum_{j=1}^d X_j \xi \otimes (\omega \wedge e_j).
    \]
    It is not hard to check that $\delta_X^2 = 0$. Moreover, $\lambda \in \sigma_T(X)$ if and only if $\ran(\delta_{X - \lambda}) \subsetneq \ker \delta_{X - \lambda}$. Now set $\alpha_X = \delta_X + \delta_X^*$. As noted in \cite{Vasilescu-characterization}, $\lambda \notin \sigma_T(X)$ if and only if $\alpha_{X - \lambda}$ is invertible. Define an operator on the differential forms in $e_1,\ldots,e_d$ and $d\bar{z}_1,\ldots,d\bar{z}_d$ with smooth $\cK$-valued coefficients, as follows
    \[
    M_X(z) \omega= \alpha_{z - X}^{-1} (\bar{\partial} \alpha_{z - X}^{-1})^{n-1} (\omega \wedge e_1 \wedge \cdots \wedge e_d).
    \]
    In particular, $M_X$ will annihilate any form of positive degree in the $e_j$. Now for a vector $\xi \in \cK$ viewed as a degree $(0,0)$ form, we have that $M_X(z) \xi$ is a form of degree $n-1$  in $d\bar{z}_1,\ldots,d\bar{z}_d$ with smooth $\cK$-valued coefficients. In particular, every coefficient is an operator that depends smoothly on $z$ applied to $\xi$. By virtue of the fact that $\partial U$ is compact, we have a constant $C > 0$ that depends only on $X$ and $U$, such that 
    \[
    \|f(X)\| \leq C \|f\|_{\infty}.
    \]
    In fact, one can take $C = \int_{\partial U} \|M_X(x)\| dS$, where $dS$ represent the surface volume element on $\partial U$. 
    Here, $\|f\|_{\infty}$ is the norm in $A(B_V)$. In particular, this homomorphism is bounded. 
    Likewise, taking a matrix $F = (f_{ij}) \in M_m(A(B_V))$ and a unit vector $\xi \in \cK^m$, we get that
    \begin{align*}
    \|F(X) \xi\| 
    &= \left\| \left(\sum_{j=1}^d f_{ij}(X) \xi_j \right)_{i=1}^d \right\| \\ 
    &= \frac{1}{(2 \pi)^d} \left\| \left(\int_{\partial U} \sum_{j=1}^d f_{ij}(z) M_X(z) \xi_j  \wedge dz_1 \wedge \cdots \wedge dz_d\right)_{i=1}^d \right\| \\ 
    &= \left\| \int_{\partial U} F(z) \big(\left[I \otimes M_X(z)\right] \xi \wedge dz_1 \wedge \cdots \wedge dz_d \big) \right\| \\
    &\leq C \|F\|_{\infty}
    \end{align*}
    thus $\|F(X)\| \leq C \|F\|_{\infty}$. 
    Therefore, the map evaluation at $X$ map is a completely bounded homomorphism from $A(B_V)$ to $B(\cK)$. 
    It follows that the composition 
    \[
    A(\bB_{\max(V)}) \to A(B_V) \to B(\cK)
    \]
    is completely bounded. By Lemma \ref{lem:specrad_homo}, this implies that $\rho_{\max(V)}(X) \leq 1$. Now arguing similarly with $(1 + \varepsilon) X$ for $\varepsilon > 0$ small, we have that $\rho_{\max(V)}(X) < 1$.
\end{proof}

\begin{remark}
The proof can be carried out using the Cauchy-Weil formula of Taylor \cite{Taylor-analytic_func_calc} and the Oka-Weil theorem for $\sigma(X)$, since the latter is polynomially convex.
\end{remark}

\begin{remark}
     By the above theorem, the spectral radius $\rho_{\cE}(X)$ of a commuting tuple $X$ depends only on the norm of $\cE$ and not on the family of matrix norms.
\end{remark}

\begin{remark}
It is known that if $\Omega \subset \bC^d$ is a bounded open set and if $M_z = (M_{z_1} \ldots, M_{z_d})$ is the tuple of multiplication operators on the Bergman space $L^2_a(\Omega)$, then for every tuple $X = (X_1, \ldots, X_d)$ of commuting operators such that $\sigma_T(X) \subset \Omega$, the tuple $X$ is jointly similar to the restriction of an ampliation of $M_z^*$ to an invariant subspace (in fact the Bergman space can be replaced by any reasonable RKHS; see \cite[Application 5.27]{Curto-applications}, which is an extension of \cite[Theorem 2]{ball1977rota}; see also \cite[Theorem 3.1]{fong1978renorming} ). 
This implies that if $X_1, \ldots, X_d$ are commuting and each $X_i$ is similar to a strict contraction , then $X_1, \ldots, X_d$ are simultaneously similar to a tuple of strict contractions\footnote{This fails if we replace ``strict contractions" with ``contractions"; see \cite[Theorem 2.1]{pisier1998joint}. On the other hand, every tuple of commuting {\em matrices} such that each is similar to a contraction, is simultaneously similar to a tuple of contractions; see \cite[Theorem 1.2]{clouatre2019joint}}. 
We can now clarify this result using Theorem \ref{thm:comm_specrad}. 

Indeed, if $X \in B(\cK^d)$ and $X_i$ is similar to a strict contraction, then $\sigma(X_i) \subset \bD$. 
If this holds for all $i=1, \ldots, d$, then by the projection property of the Taylor spectrum, $\sigma_T(X) \subset \bD^d$. 
By Theorem \ref{thm:comm_specrad}, $X$ is similar to a point in $\bD_Q^{(\infty)}$ for $Q$ a basis of any operator space $\cE$ that quantizes $\ell^\infty_d$; in particular this holds if we take $\cE = \min(\ell^\infty_d)$. 
But then a point in $\bD_Q^{(\infty)}$ is nothing but a a tuple of strict contractions, so we are done. 
Note that Theorem \ref{thm:comm_specrad} yields the stronger result that $X$ is similar to a point in the smallest quantization of the unit ball of $\ell^\infty_d$, which we may denote $\bB_{\max(\ell^\infty_d)}^{(\infty)}$. It would be interesting to understand what the operational consequences of this are.     
\end{remark}

The following simple $2 \times 2$ example shows that two non-commuting matrices, each of which is similar to a strict contraction, need not be jointly similar to a pair of strict contractions. 

\begin{example}
Consider 
\[
A = \begin{pmatrix}
    \lambda & t \\ 0 & \lambda
\end{pmatrix} 
\quad, \quad 
B = \begin{pmatrix}
    \lambda & 0 \\ t & \lambda
\end{pmatrix}, 
\]
where $\lambda \in \bD$ and $t \geq 1$ is real. 
Let $S = \left(\begin{smallmatrix} a & b \\ c & d \end{smallmatrix} \right)$ be an invertible matrix. 
Without loss of generality, we may assume that $\det(S) = 1$ and that $c = 0$ (requiring without loss that $SAS^{-1}$ is upper triangular forces $S$ to be upper triangular). 
Then, 
\[
SAS^{-1} = \begin{pmatrix}
    \lambda & ta^2 \\ 0 & \lambda
\end{pmatrix} 
\quad, \quad 
S B S^{-1} = \begin{pmatrix}
    \lambda + tbd & -tb^2 \\ td^2 & \lambda - tdb
\end{pmatrix}, 
\]
which cannot be strict contractions simultaneously, because $ad=1$ and $t \geq 1$. 
In fact, this example shows that there is no ball $B \subset M_2(\bC)^2$ such that every pair of matrices that are separately similar to a pair of strict contractions is jointly similar to pair in $B$. 
\end{example}


\subsection{Quasinilpotent tuples}\label{subsec:quasinil}

The following proposition tells us that the notion of joint quasinilpotence is independent of the operator space structure.
\begin{proposition}\label{}
    Let $\cE$ and $\cF$ be operator space structures on $\C^d$. If $A \in B(\cH)$ is such that $\rho_{\cE}(A) = 0$, then $\rho_{\cF}(A) = 0$.
\end{proposition}
\begin{proof}
    By our assumption, for each $n \in \bN$, there exists an invertible $S_n$, such that 
    \[
    \left\| \sum_{j=1}^d S_n^{-1} A_j S_n \otimes Q_j^{\cE} \right\| < 1/n.
    \]
    Since the identity map $\cE \to \cF$ is completely bounded, the induced map $B(\cH) \otimes_{\min} \cE \to B(\cH) \otimes_{\min} \cF$ is completely bounded, as well and has the same cb norm \cite[Theorem 12.3]{PaulsenBook}. Therefore, there exists $C > 0$, such that
    \[
    \rho_{\cF}(A) \leq \left\| \sum_{j=1}^d S_n^{-1} A_j S_n \otimes Q_j^{\cF} \right\| \leq C/n.
    \]
    This implies the result.
\end{proof}


\subsection{The general case}\label{subsec:general}

Our goal in this section, is to show that, in general, the spectral radius detects the operator space structure. We will show that for any two selfadjoint operator spaces $\cE_1, \cE_2$ there exist a matrix tuple $X$ such that $\rho_{\cE_1}(X) \neq \rho_{\cE_1}(X)$. In particular, this will imply that $\rho_{\min(V)}\neq\rho_{\max(V)}$ whenever the normed space $V$ is a conjugation invariant normed space. 
We then show that $\rho_{\min(V)}\neq\rho_{\max(V)}$ for any normed space of dimension $d \geq 3$. 

We begin by clarifying what we mean by {\em selfadjoint operator space}. 
If $\mathcal{E}$ is a concrete operator space in $ B(\cH)$ for some Hilbert space $\cH$, we say that $\mathcal{E}$ is \textit{selfadjoint} if $e^*\in\mathcal{E}$ for all $e\in\mathcal{E}$. 
We say that a normed space $V$ admits a selfadjoint operator space structure if there exists a concrete selfadjoint operator space $\mathcal{E}$ and a surjective isometric linear map between $V$ and $\mathcal{E}$. Equivalently, one may start with a selfadjoint operator space and then regard it only as a normed space. For example, $\mathbb{C}^d$ with any $p$-norm admits a selfadjoint structure. More generally, we will see that any norm on $\mathbb{C}^d$ that is \textit{invariant under conjugation}, i.e., 
  \[
    \|x\|=\|\bar x\| \qquad x\in \mathbb{C}^d, 
  \]
admits a selfadjoint operator space structure.


In the context of abstract operator spaces, we follow the definition in \cite{BlecherKirkNealWerner} and define a \textit{abstract selfadjoint operator space} to be an operator space $X$ with an involution $\tau:X\to X$ such that 
  \begin{equation}\label{eq: abstract op space}
    \|[\tau(x_{ji})]\|=\|[x_{ij}]\|, \qquad n\in\mathbb{N}, [x_{ij}]\in M_n(X).
  \end{equation}
If not further specified, a selfadjoint operator space means a concrete selfadjoint operator space. 
The following proposition verifies that this is indeed a suitable abstract definition of selfadjoint operator space.

\begin{proposition}\label{prop:abs selfadj opspace}
    Let $X$ be an abstract selfadjoint operator space. Then there exists a Hilbert space $\cH$ and a completely isometric map $i:X\to B(\cH)$ such that $i(X)$ is a selfadjoint operator space and $i(\tau(x))=i(x)^*$.
\end{proposition}

\begin{proof}
    Let $X$ be an abstract selfadjoint operator space, $\tau:X\to X$ an involution satisfying Equation \ref{eq: abstract op space}, and $\pi:X\to B(\cH)$ a completely isometric map. Define
      \[
        i:X\to M_2( B(\cH)), x\mapsto \begin{pmatrix} 0 & \pi(x)\\ \pi(\tau(x))^* & 0 \end{pmatrix}.
      \]
    It is obvious that $i$ is well-defined, linear and completely isometric. Moreover, $i(x)^*=i(\tau(x))$. Thus $i(X)$ is a selfadjoint operator space and the proposition proven.
\end{proof}


Note that a selfadjoint Banach space structure on $\C^d$ is $V$ equipped with an isometric antilinear involution $\tau \colon V \to V$. 

\begin{lemma}\label{lem: isometric adjoint}
    Let $V$ be a selfadjoint Banach space structure on $\C^d$ (in particular, if the norm of $V$ is conjugation-invariant). Then, for every $n \in \bN$ and every $X\in M_n(V)$,
      \[
        \|X\|_{\min(V)}=\|\tau(X)\|_{\min(V)} \qquad \textup{and} \qquad \|X\|_{\max(V)}=\|\tau(X)\|_{\max(V)}.
      \]
      Therefore, both $\min(V)$ and $\max(V)$ are selfadjoint operator spaces with the involution induced by complex conjugation.
\end{lemma}

\begin{proof}
    The lemma follows from the characterization of $\|\cdot\|_{\min(V)}$ in \cite[Theorem 14.1]{Paulsen-book} and $\|\cdot\|_{\max(V)}$ in \cite[Theorem 14.2]{Paulsen-book}.
\end{proof}

The next lemma is the starting point for the proofs of Theorem \ref{thm:selfadj_ineq} and Theorem \ref{thm:min_neq_max}, and connects the spectral radius with the underlying operator space norm. To simplify notations, we fix a completely isometric embedding of $j \colon \cE \to B(\cH)$, such that $j(\tau(x)) = j(x)^*$ and identify $\cE$ with its image. In particular, we treat $\tau(X) = X^*$ in $M_n(\cE)$. It is clear from the proof of the following lemma, that it is independent of the choice of $j$. 

\begin{lemma}\label{lem: radius=norm}
    Let $\mathcal{E}$ be a selfadjoint operator space. Then, for every $X \in M_n(\cE)$ selfadjoint $\rho_{\cE}(X) = \|X\|_{\cE}$. In particular, for every $X \in M_n(\cE)$,
      \[
      \|X\|_{\mathcal{E}}=\rho_\mathcal{E}\left(\begin{pmatrix} 0 & X\\ \tau(X) & 0\end{pmatrix}\right).
      \]
\end{lemma}

\begin{proof}
    Let $X, Y\in M_n(\mathcal{E})$. Then
      \be \label{eq:mat_norm}
        \left\|
        \begin{pmatrix}
        0 & x\\
        y & 0
        \end{pmatrix}
        \right\|_{\mathcal{E}}
        =
        \max\{\|x\|_\mathcal{E},\|y\|_\mathcal{E}\}
      \ee
    since $\|\cdot\|_\mathcal{E}$ is an operator space matrix norm. Since $\tau(X)=X^*$ for every $X \in M_n(\cE)$, the element $\begin{pmatrix} 0 & X\\ \tau(X) & 0\end{pmatrix}$ is selfadjoint. Therefore, the second part of the lemma follows from the first. Denote by $i_{\max}$ the embedding of $\cE$ into $C^*_{\max}(\cE)$ and let $A = C^*(\cE) = C^*(j(\cE))$. By the universal property of the maximal $C^*$-algebra, there exists a surjective $*$-homomorphism $\pi \colon C^*_{\max}(\cE) \to A$, such that for every $x \in \cE$,  $\pi(i_{\max}(x)) = x$. Now let $X \in M_n(\cE)$ be selfadjoint and let $\rho_A(X)$ denote the spectral radius of $X$ as an element of $M_n(A)$. In particular, $\rho_A(X) = \|X\|_{\cE}$. On the other hand, we have that $\rho_{\cE}(X) = \rho(i_{\max}(X)) \leq \|X\|_{\cE}$. However, since $\pi$ is a quotient map,
    \[
    \|X\|_{\cE} = \rho_A(X) = \rho_A(\pi(i_{\max}(X))) \leq \rho_{\cE}(X) \leq \|X\|_{\cE}.
    \]
    Therefore, we get the desired equality.

\end{proof}

As seen with the row and column spaces, two different different operator space structures can have the same spectral radii over all matrix tuples. In contrast, the next theorem shows that this is not true for selfadjoint operator space structures.

\begin{theorem} \label{thm:selfadj_ineq}
    Let $\mathcal{E}_1, \mathcal{E}_2$ be selfadjoint operator spaces, and let $i:\mathcal{E}_1\to \mathcal{E}_2$ be a linear, $*$-preserving but not completely isometric map. Then, there exist $n\in\mathbb{N}$ and $a\in M_n(\mathcal{E}_1)$ such that 
      \[
        \rho_{\mathcal{E}_1}(a)\neq\rho_{\mathcal{E}_2}(i(a)).
      \]
\end{theorem}

\begin{proof}
    Since $i$ is not completely isometric, there exist $n\in\mathbb N$ and $b\in M_n(\cE_1)$ such that
    \[
      \|b\|_{\cE_1}\neq\|i(b)\|_{\cE_2}.
    \]
    Define $a=\begin{pmatrix} 0 & b\\ b^* & 0\end{pmatrix}$. Since $i$ is $*$-preserving, we have
      \[
        i(a)=\begin{pmatrix} 0 & i(b)\\ i(b)^* & 0\end{pmatrix}.
      \]
    Lemma \ref{lem: radius=norm} therefore implies that
      \[
        \rho_{\cE_1}(a)=\|b\|_{\cE_1}\neq\|i(b)\|_{\cE_2}=\rho_{\cE_2}(i(a)),
      \]
    which proves the claim.
\end{proof}

\begin{example}
    As we saw in Lemma \ref{lem: isometric adjoint}, if $V=(\mathbb{C}^d,\|\cdot\|_V)$ is a Banach space such that pointwise complex conjugation is an isometric operation on $V$, then the minimal and maximal operator space structures over $V$ form a selfadjoint operator space via the involution given by complex conjugation. Hence, Theorem \ref{thm:selfadj_ineq} shows that there exists a tuple of matrices $X$ such that
      \[
        \rho_{\min(V)}(X)<\rho_{\max(V)}(X)
      \]
    whenever $\min(V)\neq\max(V)$. In particular, this applies to every $p$-norm on $\mathbb{C}^d$.
    In Section \ref{sec:dependence}, we will present another proof of the above result, which is of interest in its own.
\end{example}

It is time to prove that $\rho_{\min(V)}\neq\rho_{\max(V)}$ for every finite dimensional normed space $V$ of dimension at least $3$. The proof is inspired by \cite{paulsen1992representations}. We need the following lemma.

\begin{lemma}\label{lem: radii linear isomorphism}
    Let $\cE_1$ and $\cE_2$ be operator spaces which are linearly isomorphic as Banach spaces, and let
  \[
    u:\cE_1\to \cE_2
  \]
be a linear isomorphism. Suppose that $\cE_2$ is selfadjoint and that there exists
$c\in M_n(\cE_2)$ such that
  \[
    \|c\|_{\max(\cE_2)}
    >
    \|u\|\,\|u^{-1}\|\,
    \|c\|_{\min(\cE_2)}.
  \]
Then there exists $X\in M_{2n}(\cE_1)$ such that
  \[
    \rho_{\min(\cE_1)}(X)<\rho_{\max(\cE_1)}(X).
  \]
\end{lemma}

\begin{proof}
Set
  \[
    \kappa=\|u\|\,\|u^{-1}\|.
  \]
Since $\cE_2$ is selfadjoint, we can define
  \[
    a=
    \begin{pmatrix}
    0 & c\\
    c^* & 0
    \end{pmatrix}
    \in M_{2n}(\cE_2).
  \]
By Lemma \ref{lem: radius=norm},
  \[
    \rho_{\min(\cE_2)}(a)
    =
    \|a\|_{\min(\cE_2)}
    =
    \|c\|_{\min(\cE_2)},
  \]
and
  \[
    \rho_{\max(\cE_2)}(a)
    =
    \|c\|_{\max(\cE_2)}.
  \]
Hence, by assumption,
  \[
    \rho_{\max(\cE_2)}(a)
    >
    \kappa\,\rho_{\min(\cE_2)}(a).
  \]
Now put
  \[
    X=u^{-1}(a)\in M_{2n}(\cE_1).
  \]
We know that the spectral radius decreases under completely contractive maps, and for a completely bounded map $T:\cE_1\to \cE_2$ it follows that
  \[
    \rho_{\cE_2}(T(x))
    \leq
    \|T\|_{\mathrm{cb}}\rho_{\cE_1}(x).
  \]
For minimal and maximal operator space structures one has
  \[
    \|u:\min(\cE_1)\to \min(\cE_2)\|_{\mathrm{cb}}=\|u\|,
    \qquad
    \|u:\max(\cE_1)\to \max(\cE_2)\|_{\mathrm{cb}}=\|u\|,
  \]
and similarly for $u^{-1}$. 
Therefore, 
  \[
    \rho_{\min(\cE_1)}(X)
    \leq
    \|u^{-1}\|\,\rho_{\min(\cE_2)}(a),
  \]
whereas
  \[
    \rho_{\max(\cE_1)}(X)
    \geq
    \frac{1}{\|u\|}\rho_{\max(\cE_2)}(a).
  \]
Since
  \[
    \rho_{\max(\cE_2)}(a)   
    >
    \|u\|\,\|u^{-1}\|\rho_{\min(\cE_2)}(a),
  \]
we obtain
  \[
    \rho_{\max(\cE_1)}(X)
    >
    \rho_{\min(\cE_1)}(X).
  \]
This completes the proof.
\end{proof}

For a finite-dimensional Banach space $F$, define
  \[
    \alpha(F)
    =
    \sup_{m\in\mathbb N}
    \sup_{0\neq c\in M_m(F)}
    \frac{
    \|c\|_{M_m(\max(F))}
    }{
    \|c\|_{M_m(\min(F))}
    }.
  \]
We shall require the following inequality: 
\be\label{eq: alpha}
\alpha(\ell^2_d) > \sqrt{d} \quad \textrm{ for all } \, d \geq 3.
\ee
In fact, one has the well known lower bounds (see \cite[Theorem 14.3]{Paulsen-book})
\[ 
\alpha(\ell^2_d) \geq \frac{d+1}{2} \,\, \textrm{ for } d \textrm{  odd, and }\,\,  \alpha(\ell^2_d) \geq \frac{\sqrt{d^2+2d}}{2} \,\, \textrm{ for } d \textrm{ even.}
\]
For completeness, we give a brief justification for \eqref{eq: alpha} as follows. Let $L_1,\dots,L_d$ be the creation operators on the exterior algebra
$\Lambda(\mathbb C^d)$. Then the map
\[
\phi:\ell_2^d\to B(\Lambda(\mathbb C^d)),
\qquad e_i\mapsto L_i,
\]
is isometric \cite[Theorem 9.3.1]{Pisier-book}. Set
\[
c=\sum_{i=1}^d L_i^*\otimes e_i
\in M_{2^d}(\ell_2^d).
\]
Then
\[
\|c\|_{M_{2^d}(\min(\ell_2^d))}=1,
\]
whereas
\[
\|c\|_{M_{2^d}(\max(\ell_2^d))}
\geq
\left\|\sum_{i=1}^d L_i^*\otimes L_i\right\|
\geq
\sqrt{2(d-1)}, 
\]
which is strictly bigger than $\sqrt{d}$ for $d \geq 3$.

\begin{theorem} \label{thm:min_neq_max}
Let $V$ be a complex Banach space with $\dim V\geq 3$. Then there exist
$n\in\mathbb N$ and $X\in M_n(V)$ such that
  \[
    \rho_{\min(V)}(X)<\rho_{\max(V)}(X).
  \]
\end{theorem}

\begin{proof}
Let $d=\dim V\geq 3$. By the complex version of John's ellipsoid theorem (see \cite[Theorem 15.4]{TomczakJaegermann1989})
there exists an isomorphism
  \[
    u:V\to \ell_2^d
  \]
such that
  \[
    \|u\|\|u^{-1}\|\leq \sqrt d. 
  \]
By Equation \eqref{eq: alpha}, we know that
\[
\alpha(\ell_2^d)>\sqrt d,
\]
hence
\[
\frac{\|c\|_{M_{n}(\max(\ell_2^d))}}
     {\|c\|_{M_{n}(\min(\ell_2^d))}}
>
\|u\|\|u^{-1}\|, 
\]
for some $n$. 
Thus Lemma \ref{lem: radii linear isomorphism} applies with $\cE_2$ any operator space over $\ell_2^d$ (e.g. $\cE_2 = \max(\ell^2_d)$), providing an element $X\in M_{2n}(V)$ such that
\[
\rho_{\min(V)}(X)<\rho_{\max(V)}(X).
\]

\end{proof}

\begin{example}\label{ex:row_col_2}
We shall show that different non-selfadjoint operator space structures over $(\bC^2, \|\cdot\|_2)$, give, in general, different spectral radii on matrix tuples. 
For this, we consider the row norm $\|X\|_{\operatorname{row}}$, the column norm $\|X\|_{\operatorname{col}}$, and the norm 
\[
\|X\|_{\max(\operatorname{row},\operatorname{col})} = \max(\|X\|_{\operatorname{row}},\|X\|_{\operatorname{col}}). 
\]
Consider the irreducible $2$-tuple
\[
A_1 = \begin{pmatrix} 1 & 1 \\ 0 & 1 \end{pmatrix} \,\, , \,\, A_2 = \begin{pmatrix} 0 & 1 \\ 1 & 0 \end{pmatrix} , 
\]
and the corresponding induced CP map $\phi_A \colon M_2 \to M_2$, given by 
\[
\phi_A(X) = A_1 X A_1^* + A_2 X A_2^*. 
\]
By the Perron--Frobenius theory for CP maps developed by Evans and Hoegh-Krohn (see Section 2 in \cite{evans1978spectral}), there exists a unique (up to normalization) eigenvalue $H \in M_2$ such that 
\[
\phi_A(H) = rH , 
\]
with $r = \rho_{\operatorname{row}}(A)$ being the maximal eigenvalue of $\phi_A$. 
Moreover, this $H$ must be strictly positive definite. 
It is shown in \cite{evans1978spectral} that the positive $S = H^{1/2}$ is the unique minimizer of 
\[
\rho_{\operatorname{row}}(A) = \min_{K > 0} || (K^{-1}  A_1 K, K^{-1} A_2  K)||_{\operatorname{row}}.
\]
Similarly, the CP map $\Psi_A(X) = A_1^*X A_1 + A_2^* X A_2$ has maximal eigenvalue $r = \rho_{\operatorname{col}}(A) = \rho_{\operatorname{row}}(A)$ with a unique eigenvector $L$, which is positive, such that $T=L^{-1/2}$ is the unique (up to scalar) minimizer of
\[
\rho_{\operatorname{col}}(A) = \min_{K > 0} || (K^{-1}  A_1 K, K^{-1} A_2  K)||_{\operatorname{col}}.
\]
It is easy to compute the eigenvector $H = \left(\begin{smallmatrix} 2 & 1 \\ 1 & 1 \end{smallmatrix}\right)$ for $\Phi_A$ and the eigenvector $L = \left(\begin{smallmatrix} 1 & 1 \\ 1 & 2 \end{smallmatrix}\right)$ for $\Psi_A$ and one then sees that $H^{1/2} \propto \left(\begin{smallmatrix} 3 & 1 \\ 1 & 1 \end{smallmatrix}\right)$ is not co-linear with $L^{-1/2} \propto \left(\begin{smallmatrix} 3 & -1 \\ -1 & 1 \end{smallmatrix}\right)$. 
This shows that, although $r := \rho_{\operatorname{row}}(A) = \rho_{\operatorname{col}}(A)$, the unique positive operators $S = H^{1/2} , T = L^{-1/2}$ satisfying 
\[
\rho_{\operatorname{row}}(A) = || (S^{-1}  A_1 S, S^{-1} A_2  S)||_{\operatorname{row}} \quad, \quad \rho_{\operatorname{col}}(A) = || (T^{-1}  A_1 T, T^{-1} A_2  T)||_{\operatorname{col}}, 
\]
are different. 
Thus, there is no similarity that pushes $A$ to have minimal row and column norms simultaneously. This means that 
\be\label{eq:sep_maxrowcol}
\rho_{\max(\operatorname{row},\operatorname{col})}(A) > \rho_{\operatorname{row}}(A) = \rho_{\operatorname{col}}(A).
\ee
This provides a concrete example of two operator space structures $\cE_1 = \bC^2_{\operatorname{row}}$ and $\cE_2 = \bC^2_{\max(\operatorname{row},\operatorname{col})}$ quantizing the normed space $(\bC^2, \|\cdot\|_2)$ and a tuple $A \in M_n^2$ for which $\rho_{\cE_1}(A) \neq \rho_{\cE_2}(A)$, showing that the joint spectral radius depends on the operator space structure. 
\end{example}

The above example implies that $\rho_{\max(\ell^2_2)} \neq \rho_{\min(\ell^2_2)}$, which follows from Theorem \ref{thm:selfadj_ineq} and the Example right after it. 
However, it is worth noting that the separation \eqref{eq:sep_maxrowcol} does not follow from our general results. Moreover, it is worth pointing out that the uniqueness of the positive matrix that minimizes the norm on the similarity orbit is a property peculiar to the row and column balls. Take a point $(U,A) \in \overline{\fD^2}$, such that $U$ is a unitary and $A$ is a strict contraction. Then, there is a neighborhood of the identity in the cone of positive matrices, such that for each such $S$, $\|S A S^{-1}\| < 1$. Intersecting this neighborhood with the commutant of $U$, we get an infinite set of points that do not change the $\min(\ell^{\infty}_2)$ norm of the point. Moreover, it is clear because of the unitary coordinate, that the minimum of the norm on the similarity orbit is $1$.

\section{The minimal spectral radius}

In \cite{shalsham2025spectral} we defined the following alternative spectral radius function $\rho^{\min}_Q$.

\begin{definition}\label{def:min_spectral_rad}
For $Q_1, \ldots, Q_d \in B(\cH)^d$ and $X \in B(\cK)^d$, we define the {\em minimal $Q$-spectral radius} of $X$ to be
\begin{equation}\label{eq:spectral_radius_formula}
\rho^{\min}_Q(X) = \lim_{n \to \infty} \left \|\sum_{|w|=n} X^w \otimes Q_{w_1} \otimes Q_{w_2} \otimes \cdots \otimes Q_{w_n} \right\|^{1/n} . 
\end{equation}
If $\cE$ is an operator space structure on $\bC^d$ determined by $\|X\| = \sum_j X_j \otimes Q_j$ then we write $\rho^{\min}_\cE(X) = \rho^{\min}_Q(X)$
\end{definition}
The significant difference from Definition \eqref{eq:Hsr} is that here the tensor product is taken to be the minimal tensor product. 
To be precise, we take the minimal operator space tensor norm, instead of the Haagerup tensor norm, on the algebraic tensor product $\cE \otimes_{\operatorname{alg}} \cdots \otimes_{\operatorname{alg}} \cE$. 
Since the Haagerup tensor product dominates the minimal one, we always have $\rho_Q^{\min}(X) \leq \rho_Q(X)$. 

The interest in the minimal spectral radius stems from its naturality and simplicity, as well as the fact that it leads to spectral radii of interest in a couple of important cases. 

\begin{example}\label{ex:specrad_min_ellinf}
In Example 2.11 of \cite{shalsham2025spectral} it was shown that  if $Q_1, \ldots, Q_d$ are the canonical generators of the minimal operator space $\min(\ell^\infty_d)$, then $\rho^{\min}_Q(X)$ is equal to the so-called {\em Rota-Strang spectral radius}
\[
\rho^{\min}_Q(X) = \rho_{RS}(X) := \lim_{n \to \infty} \max_{|w|=n}\|W^w\|^{1/n}. 
\]
It was also noted in that same example that, when $d\geq2$, there are examples of $d$-tuples of $d \times d$ matrices $X$ such that $\rho_{RS}(X) \neq \rho_{\min(\ell^\infty_d)}(X)$.
Thus, in general, the minimal spectral radius function $\rho^{\min}_Q$ differs from the spectral radius function $\rho_Q(X)$. 
\end{example}

\begin{example}\label{ex:specrad_row_column}
On the other hand, in the case when $Q_1, \ldots, Q_d$ are the canonical generators of the row operator space $\left(\ell^2_d \right)_{\operatorname{row}}$, then $\rho^{\min}_Q(X) = \rho_Q(X)$,  (see \cite[Examples 2.10]{shalsham2025spectral}); the same holds for the column operator space. 
This follows from the defining formulas, together with the natural identifications (see \cite[p. 95]{Pisier-book})
\[
\left(\ell^2_m \right)_{\operatorname{row}} \otimes_h \left(\ell^2_m \right)_{\operatorname{row}} = \left(\ell^2_{mn} \right)_{\operatorname{row}} = \left(\ell^2_m \right)_{\operatorname{row}} \otimes_{\min}  \left(\ell^2_m \right)_{\operatorname{row}}
\]
and 
\[
\left(\ell^2_m \right)_{\operatorname{col}} \otimes_h \left(\ell^2_m \right)_{\operatorname{col}} = \left(\ell^2_{mn} \right)_{\operatorname{col}} = \left(\ell^2_m \right)_{\operatorname{col}} \otimes_{\min}  \left(\ell^2_m \right)_{\operatorname{col}}
\]
between the Haagerup and minimal tensor products for the row and column operator spaces. 
\end{example}

Thus, the minimal spectral radius may in special cases be equal to the spectral radius associated to an operator space. 
We do not know of other examples (besides the column and row spaces) in which this happens. 
However, for commuting tuples all the above mentioned spectral radii coincide. 
We now prove this result and obtain the formula \eqref{eq:spec_rad_comm}
which justifies the name ``spectral radius".

\begin{theorem} \label{thm:min_specrad}
    Let $X = (X_1,\ldots,X_d) \in B(\cK)^d$ be a commuting tuple, let $V$ be $\bC^d$ equipped with a norm $\|\cdot\|_V$, and let $\cE$ be an operator space structure on $\C^d$ quantizing $V$ determined by operators $Q_1, \ldots, Q_d \in B(\cH)$. Then, 
    \begin{equation}\label{eq:spec_rad_comm}
    \rho_Q^{\min}(X) = \rho_Q(X) = \max\{\|\lambda\|_V : \lambda \in \sigma(X) \}.
    \end{equation}
\end{theorem}

\begin{proof}
The second equality follows from Theorem \ref{thm:comm_specrad}. 
One always has $\rho_Q^{\min}(X) \leq \rho_Q(X)$, and since $\rho_{\min(V)}^{\min}(X) \leq \rho_{Q}^{\min}(X)$, it remains to show that 
\[
\rho_{\min(V)}^{\min}(X) = \max\{\|\lambda\|_V : \lambda \in \sigma(X) \}.
\]
Let $B:= \ol{B_{V^*}} \subset \bC^d$ denote the closed unit ball of the dual $V^*$. 
Represent the operator space structure $\min(V)$ on $\bC^d$ by mapping the standard basis vector $e_i \in \bC^d$ to the coordinate function $z_i$ in $C(B)$. 
The minimal tensor product $\min(V)^{\otimes_{\min}}$ is then naturally identified as the subspace of $C(B^n)$ spanned by the functions 
\[
z_{i_1} \otimes \cdots \otimes z_{i_n} \colon (x_1, \ldots, x_n) \mapsto z_{i_1}(x_1) \cdots z_{i_n}(x_n). 
\]

Referring to Equation \eqref{eq:spectral_radius_formula}, we compute 
\begin{align*}
\rho^{\min}_{\min(V)}(X) &= \lim_{n \to \infty} \left \|\sum_{|w|=n} X^w \otimes z_{w_1} \otimes \cdots \otimes z_{w_n} \right\|^{1/n} \\
&= \lim_{n \to \infty} \sup_{x \in B^n}\left \|\sum_{|w|=n}  z_{w_1}(x_1)  \cdots z_{w_n}(x_n) X^w \right\|^{1/n} \\
&= \lim_{n \to \infty} \sup_{x \in B^n} \left\|\left( \sum_{j=1}^d  z_{j}(x_1) X_j\right)  \cdots \left( \sum_{j=1}^d  z_{j}(x_n) X_j\right) \right\|^{1/n}. 
\end{align*}
Now let use consider the following set of bounded operators
\[
\cN = \left\{ \sum_{j=1}^d  z_{j}(x) X_j : x \in B\right\}, 
\]
and write $\cN^n$ for the set of all products of $n$ elements from $\cN$. 
Noting that $\cN$ is norm compact and using commutativity, it is not hard to show (and well known; see \cite{soltysiak1993joint} or \cite[Lemma 2.8]{shulman2002formulae}) that one has the following Berger-Wang type formula
\begin{equation}\label{eq:BW}
\lim_{n \to \infty} \sup_{A \in \cN^n}\|A\|^{1/n} = \lim_{n \to \infty} \sup_{A \in \cN^n}\rho(A)^{1/n} = \sup_{A \in \cN}\rho(A). 
\end{equation}
Plugging this in the above formula for $\rho^{\min}_{\min(V)}(X)$, and using the spectral mapping theorem and duality, we find 
\begin{align*}
\rho^{\min}_{\min(V)}(X) &= \sup_{x \in B}\rho\left( \sum_{j=1}^d  z_{j}(x) X_j \right) \\
&= \sup \left\{ \left|\sum_{j=1}^d x_j \lambda_j \right|  : x \in B, \lambda \in \sigma(X) \right\} \\
&= \max\left\{\|\lambda\|_V : \lambda \in \sigma(X) \right\}. 
\end{align*}
That completes the proof. 
\end{proof}

\begin{remark}
By Theorem \ref{thm:comm_specrad}, we can replace $\sigma(X)$ in the right hand side of \eqref{eq:spec_rad_comm} with $\sigma_T(X)$. 
This recovers in a somewhat roundabout and unexpected way the well-known fact that $\max\{\|\lambda\|_V : \lambda \in \sigma(X) \} = \max\{\|\lambda\|_V : \lambda \in \sigma_T(X) \}$; in fact any reasonable notion of spectrum will give the same quantity (see \cite[Corollary 10]{cho1992geometric}). 
\end{remark}

Theorem \ref{thm:min_specrad} gives a formula for the maximal norm of a point in the joint spectrum of a commuting tuple of operators $X$:
\begin{equation}\label{eq:srf}
\max\{\|\lambda\|_V : \lambda \in \sigma(X) \} = \lim_{n \to \infty} \left \|\sum_{|w|=n} X^w \otimes Q_{w_1} \otimes Q_{w_2} \otimes \cdots \otimes Q_{w_n} \right\|^{1/n}
\end{equation}
where $Q_1, \ldots, Q_d \subset B(\cH)$ are any operators such that $\sum_j x_j e_j \mapsto \sum_j x_j Q_j$ is an isometry from $V = \left(\bC^d, \|\cdot\|_V \right)$ into $B(\cH)$. 

We now compare this formula with know formulas for the $p$-norms $\|x\|_p = \left(\sum_j |x_j|^p\right)^{1/p}$ on $\bC^d$. 
In the proof above, we noted that the Berger-Wang type formula  \eqref{eq:BW}, which was noted elsewhere in the literature (\cite{shulman2002formulae,soltysiak1993joint}). 
For $\cN = \{X_1, \ldots, X_d\}$, this can be rewritten as 
\begin{equation}\label{eq:linf_srf}
\max_{\lambda \in \sigma(X)} \|\lambda\|_\infty = \lim_{n\to \infty} \max_{|w|=n}\|X^w\|^{1/n} .
\end{equation}
The right hand side of the above equation is readily seen to be equal to the right hand side of \eqref{eq:srf} for $Q_j = E_{jj}$ --- the diagonal matrix in $M_d(\bC)$ with $1$ in the $j$th diagonal slot and $0$s elsewhere, which clearly generate an isometric copy of $\ell^\infty_d$. 
As noted in \cite[Example 2.11]{shalsham2025spectral}, the quantity $\lim_{n\to \infty} \max_{|w|=n}\|X^w\|^{1/n}$ is the joint spectral radius $\rho_{RS}(X)$ introduced by Rota and Strang \cite{RS60}, and when with the above choice of $Q_1, \ldots, Q_d$ the minimal spectral radius $\rho_Q^{\min }(X)$ readily collapses to the Rota--Strang spectral radius $\rho_{RS}(X)$. 

In \cite{muller1997joint}, M\"{uller} obtained the following elegant formula for every commuting tuple of Banach space elements $X = (X_1, \ldots, X_d)$ and all $p \in [1, \infty)$: 
\begin{equation}\label{eq:srf_muller}
\max_{\lambda \in \sigma(X)}\|\lambda\|_p  = \lim_{n \to \infty} \left(\sum_{|w|=n} \|X^w\|^p \right)^{\frac{1}{np}}
= \lim_{n \to \infty} \left(\sum_{|\alpha|=n} \binom{n}{\alpha}\|X^\alpha\|^p \right)^{\frac{1}{np}}
\end{equation}
where in the middle expression $w \in \{1, \ldots, d\}$ is a word, and in the last expression $\alpha \in \bN^d$ is a multi-index and we are using multi-index notation. 
We obtain a curious identity between the right hand sides of \eqref{eq:srf} and \eqref{eq:srf_muller} in this setting.

\section{Opposite operator spaces} 

Recall that if $\cE$ is an operator space, then $\cE^{\op}$ is the operator space with matrix norms given, for $X = (x_{ij}) \in M_n(\cE)$, by 
\[
\|X\|_{\cE^{\op}} := \|X^T\|_{\cE}, 
\]
where $X^T$ denotes the transposed matrix $\left((x_{ij})_{i,j=1}^n \right)^T = (x_{ji})_{i,j=1}^n$. 
If $\cE$ is an operator space structure on $\bC^d$, then the operator space structure $\cE^{\op}$ is given concretely on matrix tuples $X = (X_1, \ldots, X_d) \in \bM^d$ by $\|X\|_{\cE^{\op}} = \|X^T\|_\cE$, where $X^T = (X_1^T, \ldots, X_d^T)$. 

The operation of transposition extends in a basis free way to operators on infinite dimensional Hilbert space. 
If $A \colon \cH \to \cK$ then $A^T \colon \cK^* \to \cH^*$ is nothing but the adjoint map between the dual spaces. 
If $A$ is represented as an infinite matrix, then $A^T$ can be represented concretely by the transposed matrix. 

Suppose that the operator space structure on $\cE$ is determined by $Q_1, \ldots, Q_d$. 
In that case the opposite $\cE^{\op}$ is determined by the tuple of operators $Q_1^T, \ldots, Q_d^T$, because 
\[
\left\| \sum_{j=1}^d X_j^T \otimes Q_j \right\| = \left\| \sum_{j=1}^d X_j \otimes Q_j^T \right\|.
\]

For the following theorem we shall need to recall the way in which the Haagerup tensor product interacts with the opposite operation. 
The Haagerup tensor norm is defined as follows: if $X \in M_n(E \otimes_{\alg} F)$, its Haagerup norm is given by 
\[
\|X\|_{M_n(E\otimes_h F)} = \min\{\|Y\|\|Z\| :  X = Y \odot Z \},
\]
where the minimum is over all $Y \in M_{n,m}(E)$ and $Z \in M_{m,n}(F)$ and the product $Y \odot Z$ is defined by $(Y \odot Z)_{ij} = \sum_k Y_{ik} \otimes Z_{kj}$. 
Now, if $X = Y \odot Z$ as above, then 
\[
(X^T)_{ij} = X_{ji} = \sum_k Y_{jk} \otimes Z_{ki}. 
\]
Now if $\cF \colon F \otimes_{\alg} E \to E \otimes_{\alg} F$ is the flip $y \otimes x \mapsto x \otimes y$, we see that 
\[
(X^T)_{ij} = \cF \left( (Z^T \odot Y^T)_{ij} \right). 
\]
Writing $\cF$ also for the flip $E \otimes_{\alg} F \to F \otimes_{\alg} E$, we conclude that 
\be\label{eq:XYZ}
X = Y \odot Z \textrm{ if and only if } \cF(X^T) = Z^T \odot Y^T .
\ee
Now,
\be\label{eq:norm1}
\|X\|_{M_n((E \otimes_h F)^{\op})} = \|X^T\|_{M_n(E \otimes_h F)} = \min\left\{\|Y\|_E \|Z\|_F : X^T = Y \odot Z \right\}, 
\ee
while, on the other hand, 
\begin{equation}
\begin{aligned}[t]
\|\cF(X)\|_{M_n(F^{\op} \otimes_h E^{\op})} 
&= \min\left\{ \|Z\|_{F^{\op}} \|Y\|_{E^{\op}} : \cF(X) = Z \odot Y \right\} \\
&= \min\left\{ \|Z^T\|_F \|Y^T\|_E : \cF(X) = Z \odot Y \right\}. 
\end{aligned}
\label{eq:norm2}
\end{equation}
Combining \eqref{eq:XYZ}, \eqref{eq:norm1} and \eqref{eq:norm2}, and noting that 
\[
\min\left\{\|Y\|_E \|Z\|_F : X^T = Y \odot Z \right\} = \min\left\{\|Y^T\|_E \|Z^T\|_F : X^T = Y^T \odot Z^T \right\},
\]we see that the flip induces a completely isometric isomorphism 
\[
(E \otimes_h F)^{\op} \cong F^{\op} \otimes_h E^{\op}. 
\]
Similarly, one can show that the flip induces a completely isometric isomorphism
\[
(E_1 \otimes_h E_2 \otimes_h \cdots \otimes_h E_n)^{\op} \cong E_n^{\op} \otimes_h \cdots \otimes_h E_2^{\op} \otimes_h E_1^{\op}. 
\]

\begin{theorem}\label{thm:op}
    Let $\cE$ be an operator space structure on $\bC^d$ and let $\cE^{\op}$ be the opposite operator space structure. 
    Then 
    \[
    \rho_{\cE^{\op}}(X) = \rho_{\cE}(X^T), 
    \]
    for all $X \in B(\cK)^d$.
\end{theorem}
\begin{proof}
We suppose that the operator space structure on $\cE$ is determined by $Q_1, \ldots, Q_d$ and that the opposite $\cE^{\op}$ is determined by the tuple of operators $Q_1^T, \ldots, Q_d^T$.
\begin{align*}
\rho_{\cE^{\op}}(X) &= \lim_{n \to \infty} \left \|\sum_{|w|=n} X^w \otimes Q^T_{w_1} \otimes_h Q^T_{w_2} \otimes_h \cdots \otimes_h Q^T_{w_n} \right\|^{1/n} \\
&= \lim_{n \to \infty} \left \|\sum_{|w|=n} (X^w)^T \otimes Q_{w_n} \otimes_h \cdots \otimes_h  Q_{w_2} \otimes_h Q_{w_1} \right\|^{1/n} \\
&= \lim_{n \to \infty} \left \|\sum_{|w|=n} X_{w_n}^T \cdots X_{w_1}^T \otimes Q_{w_n} \otimes_h \cdots \otimes_h  Q_{w_2} \otimes_h Q_{w_1} \right\|^{1/n} \\
&= \rho_\cE(X^T), 
\end{align*}
as claimed. 
\end{proof}

\begin{example}[The row and column operator spaces]\label{ex:row_col_equal}
Let $\cE = \left(\ell^2_d\right)_{\operatorname{row}}$ be the row operator space structure on $\bC^d$ generated by the standard basis elements $Q_1 = E_{11}, Q_2 = E_{12}, \ldots , Q_d = E_{1d}$. Then $\cE^{\op}$ is the column operator space structure on $\bC^d$ with basis $Q_1^T = E_{11}$, $Q_2^T = E_{21}$, $\ldots$, $Q_d^T = E_{d1}$. 
In Example \ref{ex:specrad_row_column}, we noted that the spectral radius in these cases is equal to the minimal spectral radius. 
Thus
\begin{align*}
\rho_{\operatorname{row}}(X) := \rho_\cE(X) &= \lim_{n \to \infty} \left\|\sum_{|w|=n} X^w \otimes E_{11} \otimes E_{12} \otimes \cdots \otimes E_{1d} 
\right\|^{1/n} \\
&= \lim_{n \to \infty} \left\|\operatorname{row} \left( X^w \right)_{|w|=n} \right\|^{1/n} \\
&= \lim_{n\to \infty} \left\| \sum_{|w|=n} X^w X^{w*} \right\|^{1/2n} 
\end{align*}
and
\begin{align*}
\rho_{\operatorname{col}}(X) := \rho_{\cE^{\op}}(X) &= \lim_{n \to \infty} \left\|\sum_{|w|=n} X^w \otimes E_{11} \otimes E_{21} \otimes \cdots \otimes E_{d1} 
\right\|^{1/n} \\
&= \lim_{n \to \infty} \left\|\operatorname{col} \left( X^w \right)_{|w|=n} \right\|^{1/n} \\
&= \lim_{n\to \infty} \left\| \sum_{|w|=n} X^{w*} X^w \right\|^{1/2n} .
\end{align*}
Plugging $X^T$ instead of $X$ in the above formula for $\rho_{\operatorname{row}}$ (and using the invariance of the operator norm under complex conjugation), we verify that $\rho_{\cE^{\op}}(X) = \rho_{\cE}(X^T)$ holds as we know it should by Theorem \ref{thm:op}. 

For the row and column operator spaces there is an additional identity 
\be\label{eq:rho_row_rho_col}
\rho_{\operatorname{row}}(X) = \rho_{\operatorname{col}}(X) \,\, , \,\, \textrm{ for all } X \in \bM^d. 
\ee
This was proved in \cite[Lemma 2.1]{JMS-clark}, but we provide a short proof for completeness. 
Recall that for a $k \times k$ positive matrix $P$ we have the simple relations $\|P\| \leq \Tr(P) \leq k \|P\|$. 
Then for every $X \in M_k^d$ and every $n$, 
    \begin{align*}
    \left\| \sum_{|w|=n} X^{w*} X^w \right\|
    & \leq \Tr\left( \sum_{|w|=n} X^{w*}X^w \right) \\
    &= \Tr\left( \sum_{|w|=n} X^wX^{w*} \right) \\
    & \leq k \left\|\sum_{|w|=n} X^wX^{w*} \right\|. 
    \end{align*}
    Taking $2n$-th roots followed by a limit we find that $\rho_{\operatorname{col}}(X) \leq \rho_{\operatorname{row}}(X)$. The reverse inequality is shown in the same way, yielding \eqref{eq:rho_row_rho_col}. 
    However, the row and column spectral radii do not coincide, as can be seen by considering a tuple $V = (V_1, \ldots, V_d)$ of isometries with pairwise orthogonal ranges (that is, a row isometry). 
    An easy calculation shows that
    \be\label{eq:VVstar}
    \rho_{\operatorname{row}}(V) = \lim_{n\to \infty} \left\| \sum_{|w|=n} V^w V^{w*} \right\|^{1/2n} = 1, 
    \ee
    while 
    \be\label{eq:VstarV}
    \rho_{\operatorname{col}}(V) = \lim_{n\to \infty} \left\| \sum_{|w|=n} V^{w*} V^{w} \right\|^{1/2n} = 
    \lim_{n\to \infty} (d^n)^{1/2n}= \sqrt{d}. 
    \ee
    Thus $\rho_{\operatorname{row}}(V) \neq \rho_{\operatorname{col}}(V)$. 
    This example also shows, significantly, that the spectral radii associated with two operator spaces might agree on all matrix tuples yet still give different values on operator tuples. 
\end{example}

\begin{remark}
There is another way to show \eqref{eq:rho_row_rho_col} which uses the characterization of the row and column spectral radii as the spectral radius of a CP map (this point of view was used in \cite{JMS-clark,Pascoe21,SSS2} and elsewhere). 
If $X \in M_n^d$ we define $\Phi_X, \Psi_X \colon M_n \mapsto M_n$ by $\Phi_X(T) = \sum X_j T X_j^*$ and $\Psi_X = \sum X_j^* T X_j$. 
Now, since the norm of a CP map is attained on the identity, we see that 
\[
\rho_{\operatorname{row}}(X)= \lim_{n\to \infty} \left\| \sum_{|w|=n} X^w X^{w*} \right\|^{1/2n} = \lim_{n \to \infty} \|\Phi_X^n\|^{1/2n} = \rho(\Phi_X)^{1/2}
\]
and likewise $\rho_{\operatorname{col}}(X) = \rho(\Psi_X)^{1/2}$. 
Using the fact that $\Phi_X$ considered as map on $M_n$ with the operator norm, is the Banach space adjoint map $\Psi_X^*$ of $\Psi_X$ considered as a map on $M_n$ with the trace norm, together with the equivalence of all norms on $M_n$, we see that $\rho(\Phi_X) = \rho(\Psi_X^*) =\rho(\Psi_X)$ and therefore $\rho_{\operatorname{row}}(X) = \rho_{\operatorname{col}}(X)$. 
This argument breaks down if the tuple $X$ acts on an infinite dimensional space (because the trace norm and the operator norm are not equivalent) and, indeed, the above example shows that the column and row spectral radii need not coincide in that case. 
\end{remark}

\section{Spectral radii and domains of operator realizations}\label{sec:rad_pencil}

A {\em realization} is a triple $(A,b,c)$ consisting of a tuple $A = (A_1, \ldots, A_d) \in B(\cH)^d$, and vectors $b,c \in \cH$. 
The terminology comes from the fact that every realization gives rise to a NC function $f = f_{(A,b,c)}$ defined in a neighborhood of $0 \in \bM^d$ given for $X \in M_n^d$ sufficiently close to $0$ by the formula
\be\label{eq:f_realization}
f(X) = (I_n \otimes b)^* \left(I_n \otimes I_\cH - \sum_{j=1}^d X_j \otimes A_j \right)^{-1} (I_n \otimes c), 
\ee
and, conversely, for every uniformly NC function $f$ defined in a neighborhood of $0 \in \bM^d$ there exists a realization $(A,b,c)$ such that $f = f_{(A,b,c)}$; see \cite{augat2025operator} and the references therein. 
The space $\cH$ can be chosen to be finite dimensional if and only if $f$ is a NC rational function. 

The terminology {\em operator realization} is used to emphasize that one might be considering infinite dimensional $\cH$ in the realization. 
Unlike in the well-studied case of NC rational functions with finite dimensional realization, in the general case of operator realizations there does not exists a unique-up-to similarity minimal realization; rather, there is a notion of minimal realization but any two minimal realization are related by the rather weak notion of {\em pseudo-similarity}; see \cite[Theorem 3.6]{augat2025operator}. 

Given an operator realization $(A,b,c)$, we define the {\em domain} of the realization to be the NC domain.
\[
\dom(A,b,c) = \left\{X \in \bM^d : I - \sum_j X_j \otimes A_j \textrm{ is invertible } \right\}; 
\]
in other words, $\dom(A,b,c)$ is the domain of definition of the NC function $z \mapsto b^*(1 - z A)^{-1}c$.
We also define 
\[
\dom^{(\infty)}(A,b,c) = \dom(A,b,c) \sqcup \left\{X \in B(\cK)^d : I - \sum_j X_j \otimes A_j \textrm{ is invertible } \right\}, 
\]
where $\cK$ is some fixed infinite dimensional separable Hilbert space. 
The notation stresses that the above notions of domain are attached to the realization rather than the NC function that it gives rise to. We also define the $B(\cH)$-valued pencil 
\[
L_A(X) = I - \sum_j X_j \otimes A_j , 
\]
and we define
\[
\dom(L_A^{-1}) = \dom(A,b,c) \quad \textrm{ and } \dom^{(\infty)}(L_A^{-1}) = \dom^{(\infty)}(A,b,c).
\]

In \cite[Theorem 3.4]{shalsham2025spectral}, it was shown that if $\cE$ is an operator space structure on $\bC^d$ and $E = \cE^*$ is the operator space dual, and if $(A,b,c)$ is a (minimal) realization for a NC rational function $f$, then $\rho_E(A)<1$ if and only if there exists $r>1$ such that $r\bB_\cE$ is contained in $\dom(A,b,c)$. 
It was also shown that these conditions are also equivalent to the statment that $f$ belongs to the algebra $A(R\bB_\cE)$ for all $R<1$. 
Our next goal is to prove a counterpart of \cite[Theorem 3.4]{shalsham2025spectral} for operator realizations.

\begin{theorem}\label{thm:rVSdom}
    Given an operator space structure $\cE$ on $\bC^d$, let $E = \cE^*$ denote the dual operator space structure on $\bC^d$. 
    For $A \in B(\cH)^d$ and $r > 0$, the following are equivalent:
    \begin{enumerate}
        \item $\rho_E(A) < r^{-1}$. \label{it:rho_lt_r}
        \item $R\bB_\cE^{(\infty)}\subset \dom^{(\infty)}(L^{-1}_A)$ for some $R > r$. \label{it:RBEinf}
        \item The similarity envelope of $R\bB^{(\infty)}_\cE$ is contained in $\dom^{(\infty)}(L^{-1}_A)$ for some $R > r$.\label{it:RBEinf_sim}
        \item For some $R > r$ the pencil $L_A$ is invertible in $A(R\bB_\cE) \otimes B(\cH)$. \label{it:inABE}
    \end{enumerate}
    In particular, for all $X \in B(\cK)^d$, if $\rho_\cE(X) \rho_E(A) < 1$ then $X \in \dom(L_A^{-1})$.
\end{theorem}
\begin{proof}
    The equivalence of items (\ref{it:RBEinf}) and (\ref{it:RBEinf_sim}) is clear and will not be remarked upon further. We will begin by proving the final assertion of the theorem.
    
    Suppose that $X \in B(\cK)^d$ and that $\rho_\cE(X) \rho_E(A) < 1$. 
    Then, by Theorem \ref{thm:SS25}, there exist invertible operators $S,T$ such that $\|S^{-1} X S\|_{\cE} \|T^{-1}AT\|_E < 1$. 
    By the basic operator space-theoretic inequality $\|\sum_j a_j \otimes b_j \| \leq \|a\|_\cE \|b\|_E$ (see \cite[Equation (2.11.2)]{PisierBook}), it follows that 
    \[
    \left\|\sum_{j=1}^d S^{-1} X_j S \otimes T^{-1}A_jT \right\| \leq \|S^{-1} X S\|_{\cE} \|T^{-1}AT\|_E < 1 .
    \]
    Thus, the operator 
    \[
    I - \sum_{j=1}^d S^{-1} X_j S \otimes T^{-1}A_jT = (S \otimes T)^{-1} \left(I - \sum_{j=1}^d X_j \otimes A_j \right)(S \otimes T)
    \]
    is invertible, and so $I - \sum_{j=1}^d X_j \otimes A_j$ is invertible, too, whence $X \in \dom^{(\infty)}(L^{-1}_A)$. 
    
    We now turn prove that item (\ref{it:rho_lt_r}) is equivalent to items (\ref{it:RBEinf}) and (\ref{it:RBEinf_sim}). 
    It follows from the above paragraph that if $\rho_E(A) < r^{-1}$ then the similarity envelope of ${r\bB_\cE^{(\infty)}}$ is contained in $\dom^{(\infty)}(L^{-1}_A)$. 
    But if $\rho_E(A) < r^{-1}$ then $\rho_E(A) < R^{-1}$ for some $R > r$, and we obtain that the similarity envelope of ${R\bB_\cE^{(\infty)}}$ is contained in $\dom^{(\infty)}(L^{-1}_A)$. 

    Suppose now that $R\bB_\cE^{(\infty)}$ is contained in $\dom^{(\infty)}(L^{-1}_A)$ for some $R > r$. 
    Let $t \in \bC\setminus \{0\}$ such that $|t| < R$. 
    Let $Q_1, \ldots, Q_d$ denote the basis that induces the operator structure $\cE$ on $\bC^d$ and let $Z_1, \ldots, Z_d$ denote the dual basis that induces the operator space structure $E = \cE^*$ on $\bC^d$. 
    Think of $Z_1, \ldots, Z_d$ as elements in $C^*_{\max}(E)$.
    We have that $\|\sum_j Z_j \otimes Q_j\| \leq 1$ and therefore $(t Z_1, \ldots, tZ_d) \in R\bB_\cE^{(\infty)} \subset \dom(L^{-1}_A)$. 
    This means that $1 - t \sum_j Z_j \otimes A_j$ is invertible in $C^*_{\max}(E) \otimes B(\cH)$, therefore $t^{-1}$ is not in the spectrum of $\sum Z_j \otimes A_j$.  
    Since this holds for all $0<|t|<R$, we have that the spectrum of $\sum Z_j \otimes A_j$ is contained in $R^{-1}\ol{\bD}$, or
    \[
    \rho_E(A) = \rho\left(\sum Z_j \otimes A_j\right) \leq R^{-1} < r^{-1}, 
    \]
    as required. 
    
    It remains to show that item (\ref{it:rho_lt_r}) is equivalent to item (\ref{it:inABE}). 
    If $\rho_E(A) < r^{-1}$, then $\rho_E(A) < r_1^{-1}$ for some $r_1>r$, and so by \eqref{eq:Hsr} we have that the Neumann series $\sum_{n=0}^\infty \sum_{|w|=n} (RZ)^w \otimes A^w$ converges in norm for every $R \in (r,r_1)$, and so we obtain that $L_A$ is invertible in $A(R\bB_\cE) \otimes B(\cH)$ for some $R>r$. 
    Conversely, suppose that there is some $R > r$ such that $L_A$ has an inverse $A(R\bB_\cE) \otimes B(\cH)$. 
    Then, for every $r_0 \in (r,R)$, we know by \cite{KVV} that the $B(\cH)$-valued NC function $L_A(Z)^{-1} = \left(1 \otimes I_\cH - \sum_j Z_j \otimes A_j \right)^{-1}$ has a boundedly convergent Taylor--Taylor series in $r_0 \bB_\cE$. But we know that near the origin of $\bM^d$ this Taylor--Taylor series is given by 
    \be\label{eq:TTseries}
    \sum_{n=0}^\infty \sum_{|w|=n} Z^w \otimes A^w, 
    \ee 
    and so by uniqueness of the Taylor--Taylor series, the series \eqref{eq:TTseries} converges boundedly in $r_0\bB_\cE$, which means that
    \[
    \left\| r_0^n \sum_{|w|=n} Z^w \otimes A^w \right\| \leq C
    \]
    for some $C>0$, whence 
    \[
    \rho_E(A) = \lim_{n \to \infty} \left\|\sum_{|w|=n} Z^w \otimes A^w \right\|^{1/n} \leq r_0^{-1} < r^{-1}.  
    \]
\end{proof}

\begin{corollary}
    Let $(A,b,c)$ be an operator realization. Then if any one of the equivalent conditions in Theorem \ref{thm:rVSdom} holds, then the function $f$ defined by \eqref{eq:f_realization} is in $A(R\bB_\cE)$. 
\end{corollary}

\begin{corollary}
    \[
    \sup\left\{R : R\bB^{(\infty)}_\cE \subset \dom(L_A^{-1})\right\} = \frac{1}{\rho_E(A)} . 
    \]
\end{corollary}

In \cite[Theorem 3.4]{shalsham2025spectral} the equivalence of items (\ref{it:rho_lt_r}) and (\ref{it:inABE}) from Theorem \ref{thm:rVSdom} was shown in the case that $\dim \cH < \infty$, and it was also shown that in that case they are also equivalent to the following two conditions: 
\begin{itemize}
          \item[$(2')$] $R\bB_\cE\subset \dom(A,b,c)$ for some $R > r$. \label{it:RBE}
        \item[$(3')$] The similarity envelope of $R\bB_\cE$ is contained in $\dom(A,b,c)$ for some $R > r$. \label{it:RBE_sim}
\end{itemize}
The following example shows that this equivalence does not persist in the case that $A$ is an operator tuple. 

\begin{example}\label{ex:not_persist}
    Fix $d \geq 2$. We  define the operator space structures
    \[
    \cE = \left(\ell^2_d\right)_{\operatorname{row}} \quad \quad \quad \textrm{ (the row operator space),}
    \]
    and 
    \[
    E = \cE^* = \left(\ell^2_d\right)_{\operatorname{col}} \quad \quad \textrm{ (the column operator space).}
    \]
    Let $\cH = \cF_d^2 := \oplus_{n=0}^\infty (\bC^d)^{\otimes n}$ be the full Fock space on $d$ generators. 
    Fix an orthonormal basis $\{e_1, \ldots, e_d\}$ for $\bC^d$, and let $A = L$ be the full shift on the Fock space given by $A_j = L_j$ for $j=1, \ldots, d$, where
    \[
    L_j e_{i_1} \otimes \cdots \otimes e_{i_n} = 
    e_j \otimes e_{i_1} \otimes \cdots \otimes e_{i_n}, 
    \]
    for any $i_1, \ldots, i_n \in \{1, \ldots, d\}$. 
    To show that, for this choice of $A$, condition $(2')$ (and therefore  also $(3')$) is not equivalent to conditions (\ref{it:RBEinf}) and (\ref{it:rho_lt_r}) of Theorem \ref{thm:rVSdom}, we to find $r>0$ and $R>r$ such that $R\bB_\cE \subset \dom(L^{-1}_A)$ but $R\bB^{(\infty)}_\cE \nsubseteq \dom^{(\infty)}(L^{-1}_A)$; equivalently, we need to find $r>0$ and $R>r$ such that $R\bB_\cE \subset \dom(L^{-1}_A)$ such that $\rho_E(A) \nless r^{-1}$. 

    To that end, choose $r$ such that ${1}/{\sqrt{d}} < r < 1$, and let $R=1$. We first show that $R\bB_\cE = \bB_\cE \subset \dom(L^{-1}_A)$. 
    Since $\cE$ is the row operator space, $\bB_\cE$ is nothing but the familiar row ball $\fB_d$ defined in \eqref{eq:Ball} consisting of strict row contractions. 
    Now let $X\in \bB_\cE$, and put $\|X\|_{\operatorname{row}} = \|\sum X_j X_j^*\|^{1/2} =: t$. 
    Then 
    \[
    \sum_{j=1}^d X_j X_j^* \leq t^2 I, 
    \]
    and by induction we obtain for all $n$ 
    \[
    \sum_{|w|=n}^d X^w X^{w*} \leq t^{2n} I. 
    \]
    Now, we use the fact that $A$ is a row isometry to find
    \[
    \left(\sum_{|w|=n} X^w \otimes A^w\right)^*\left(\sum_{|w|=n} X^w \otimes A^w\right)  = \sum_{|w|=n} X^{w*}X^w \otimes A^{w*} A^w = \sum_{|w|=n} X^{w*}X^w \otimes I 
    \]
    and therefore
    \be\label{eq:sumwxa}
    \left\|\sum_{|w|=n} X^w \otimes A^w \right\| = \left\| \sum_{|w|=n} X^{w*}X^w \otimes I \right\|^{1/2} = \left\| \sum_{|w|=n} X^{w*}X^w \right\|^{1/2}. 
    \ee
    Suppose that the $X$ we started with lies in $\bB_\cE(k)$, that is, $X$ is a tuple of $ k \times k$ matrices. 
    For $B \in M_k(\bC)$ we have the inequalities $\|B\|^2 \leq \Tr(M^*M) \leq k \|M\|^2$, thus
    \begin{align*}
    \left\|\sum_{|w|=n} X^w \otimes A^w \right\|^2
    & \leq \Tr\left( \sum_{|w|=n} X^{w*}X^w \right) \\
    &= \Tr\left( \sum_{|w|=n} X^wX^{w*} \right) \\
    & \leq k \left\|\sum_{|w|=n} X^wX^{w*} \right\| \leq k t^{2n}. 
    \end{align*}
    Therefore, the Neumann series $\sum_{n=0}^\infty \sum_{|w|=n} X^w \otimes A^w$ converges, which means that $L_A(X)$ is invertible, or in other words $X \in \dom(L_A^{-1})$. 
    We conclude that $\bB_\cE \subset \dom(L^{-1}_A)$. 

    Next, we shall find an element in $\bB_\cE^{(\infty)}$ which is not in the operator domain $\dom^{(\infty)}(L_A^{-1})$ of $L_A^{-1}$. 
    Let $V = (V_1, \ldots, V_d)$ be a row isometry, that is, 
    \[
    V_i^* V_j = \delta_{ij}I \,\, , \,\, \textrm{ for all } i,j=1, \ldots, d. 
    \]
    In particular, $V$ is a row contraction. So $Y = \frac{1}{\sqrt{d}} V \in \bB_\cE^{(\infty)}$. 
    Then 
    \begin{align*}
    \left(\sum_{j=1}^d Y_j \otimes A_j\right)^* \left(\sum_{j=1}^d Y_j \otimes A_j\right) 
    &= \frac{1}{d}\sum_{i,j=1}^d V_i^* V_j \otimes A_i^* A_j \\
    &=\frac{1}{d} \sum_{j=1}^d I \otimes I = I, 
    \end{align*}
    meaning that $\sum_{j=1}^d Y_j \otimes A_j$ is an isometry. 
    Since this isometry is not surjective, the spectrum of $\sum_{j=1}^d Y_j \otimes A_j$ contains the unit circle, therefore $L_A(Y) = I - \sum_{j=1}^d Y_j \otimes A_j$ is not invertible. 
    So $Y \in \frac{1}{\sqrt{d}}\ol{\bB^{(\infty)}_\cE}$ while $Y \notin \dom^{(\infty)}(L_A^{-1})$, which entails that $R \bB^{(\infty)}_\cE \nsubseteq \dom^{(\infty)}(L_A^{-1})$ for any $R > r > 1/\sqrt{d}$, as we wanted to show. 

    Finally, since $A$ is a row isometry, then as we showed in Example \ref{ex:row_col_equal} (see Equation \eqref{eq:VstarV}) we have that $\rho_E(A) = \sqrt{d}$. 
    This means that for every $r \in (1/\sqrt{d},1)$, we have for some $R > r$ that $R\bB_\cE \subset \dom(L_A^{-1})$, while at the same time $\rho_E(A) \nless r^{-1}$, showing that condition $(2')$ stated before the example is not equivalent to the first item in Theorem \ref{thm:rVSdom}.
\end{example}

\begin{remark}
    The phenomenon observed is in the previous example is a multivariable phenomenon: in the case $d = 1$ the matrix ball $\bB_\cE$ can replace the operator-level ball $\bB^{(\infty)}_\cE$ in item (\ref{it:RBEinf}) of Theorem \ref{thm:rVSdom}. Indeed, in the case $d = 1$ then up to scaling $\cE = E = \bC$, so if $A \in B(\cH)$ and $I - X \otimes A$ is invertible for all $X \in \bB_\cE$, then $I - zA$ is invertible for all $z \in \bD$, and this readily shows that the spectrum of $A$ is in $\ol{\bD}$, i,e, $\rho(A) \leq 1$. 
\end{remark}

\begin{remark}
    The above example is analogous to the one in \cite{Pascoe20} in the following sense. The pencil $L_A(X) = I - \sum_{j=1}^d X_j \otimes A_j$ is invertible on the unit ball $\bB_\cE = \fB_d$, while its inverse is bounded only on sub-balls of the ball of radius $1/\sqrt{d}$. The reason for the latter claim is Theorem \ref{thm:rVSdom}. It can also be verified directly using the comparison between the row and column norms. On the other hand, we claim that the inverse of the pencil is unbounded on $R\bB_{\cE}$ for every $1/\sqrt{d} \leq R < 1$. To see this take $B^{(n)} = \frac{1}{\sqrt{d}} P_n L|_{\cF^2_{d,n}}$, where $\cF^2_{d,n}$ is the subspace of the full Fock space spanned by the monomials of degree at most $n$ and $P_n$ the orthogonal projection onto $\cF^2_{d,n}$. Since $\cF^2_{d,n}$ is $L$-coinvariant, this compression is a homomorphism. Therefore, $B^{(n)}$ is jointly nilpotent of order $n+1$. In particular,
    \[
    \left(I - \sum_{j=1}^d B_j^{(n)} \otimes A_j\right)^{-1} = I + \sum_{k=1}^n \sum_{|w|=k} (B_j^{(n)})^w \otimes A^w = I + \sum_{k=1}^n d^{-k/2} \sum_{|w|=k} P_n L^w|_{\cF^2_{d,n}} \otimes L^w.
    \]
    Now let $1 \in \cF^2_d$ be the vacuum vector. Then,
    \[
    \left\| \left(I - \sum_{j=1}^d B_j^{(n)} \otimes A_j\right)^{-1}(1 \otimes 1) \right\|^2 = n + 1.
    \]
    Therefore, the inverse is not bounded on the any closed  row ball of radius at least $1/\sqrt{d}$. 
\end{remark}

Unlike \cite[Theorem 3.4]{shalsham2025spectral}, Theorem \ref{thm:rVSdom} focuses on the domain of $L_A^{-1}$ rather than on the domain of the NC function that has realization $(A,b,c)$. 
The reason for this is that operator realizations are, in general, non-unique, and that even under a minimality assumption, the domain of the pencil appearing in the realization is not the same as the domain of the function. 
This phenomenon arises already in the one variable case, as illustrated by the following example, which is borrowed from \cite{augat2025operator}. 

\begin{example}[Example 3.10 in \cite{augat2025operator}]
The entire function $f(z) = \frac{e^z-1}{z}$ has two different minimal operator realizations $(A,b,c)$; the first given by $\cH = H^2(\bD)$ --- the Hardy space on the unit disc --- and $(A,b,c) = (S^*,1,f)$ where $S$ is the forward shift, and the second given by $\cH = L^2[0,1]$ and $(A,b,c) = (T,1,1)$ where $T$ is the Volterra operator. The domain of $f$ is $\bC$, and it lifts to a NC function whose domain contains all bounded operators. 
The spectral radius of $S^*$ is equal to $1$, so the largest disc contained in the domain of $(I - zS^*)^{-1}$ is the open unit disc, while the spectral radius of $T$ is $0$, and the pencil $L_T(X) = I - X \otimes T$ is invertible for all operators $X$. 
\end{example}

\section{Dependence on operator space structure: a NC function theoretic approach}\label{sec:dependence}

In this section we will apply Theorem \ref{thm:rVSdom} to provide an alternative proof of (a weaker version of) Theorem \ref{thm:selfadj_ineq}. 
We begin by setting some notation, to establish a concrete perspective on selfadjoint operator spaces as discussed in Section \ref{subsec:general}. 
For $x = (x_1, \ldots, x_d) \in \bC^d$ we write $\ol{x} = (\ol{x_1}, \ldots, \ol{x_d})$. 
If $\cE$ is an operator space structure over $\bC^d$ and $X = (X_1, \ldots, X_d) \in M_n(\cE)$, then we define $X^* = (X_1^*, \ldots, X_d^*)$. 
We say that a normed space $V = (\bC^d, \|\cdot\|)$ is {\em invariant under conjugation} if $\|x\| = \|\ol{x}\|$ for all $x \in X$. 
Likewise, we say that an operator space $\cE$ is {\em invariant under taking adjoints} if $\|X\| = \|X^*\|$ for all $n$ and all $X \in M_n(E)$. 

Most familiar norms on $\bC^d$ are invariant under conjugation; for example, $p$-norms and operator norms induced by conjugation invariant norms. Recall that it follows from Lemma \ref{lem: isometric adjoint} that for every such $V$, $\min(V)$ and $\max(V)$ are selfadjoint. We also note that as we have seen in the proof of Lemma \ref{lem: radius=norm}, if $\cE$ is an operator space structure on $\bC^d$ determined by a basis $Q_1, \ldots, Q_d$ that is invariant under taking adjoints, then for every $X \in M_n^d$, 
    \be\label{eq:norm_adj}
    \left\| \begin{pmatrix} 0 & X \\ X^* & 0 \end{pmatrix} \right\|_{\cE} = \|X\|_\cE. 
    \ee

In the following proposition we shall need to consider the selfadjoint slice $\bB^{sa}_{\cE}$ of a NC operator unit ball $\bB_\cE$ defined as follows
\[
\bB_{\cE}^{sa} = \{X \in \bB_\cE : X = X^* \}. 
\]
\begin{lemma}\label{lem:sa_equal}
    Let $\cE_1$ and $\cE_2$ be two operator space structures on $\bC^d$ that are invariant under taking adjoints. 
    If $\bB^{sa}_{\cE_1} \subset \bB^{sa}_{\cE_2}$, then $\bB_{\cE_1} \subset \bB_{\cE_2}$, and consequently if $\bB^{sa}_{\cE_1} = \bB^{sa}_{\cE_2}$ then $\cE_1 = \cE_2$. 
\end{lemma}
\begin{proof}
    If $X \in \bB_{\cE_1}$, then $\left(\begin{smallmatrix}
        0 & X \\ X^* & 0 
    \end{smallmatrix}\right) \in \bB_{\cE_1}^{sa}$, so by the assumption $\bB_{\cE_1}^{sa} \subset \bB_{\cE_2}^{sa}$ we have $\left(\begin{smallmatrix}
        0 & X \\ X^* & 0 
    \end{smallmatrix}\right) \in \bB_{\cE_2}^{sa}$. 
    Therefore, $X \in \bB_{\cE_2}$, therefore $\bB_{\cE_1} \subset \bB_{\cE_2}$, as required. 
    The final assertion follows immediately since $\bB_{\cE_1} = \bB_{\cE_2}$ is the same thing as $\cE_1 = \cE_2$. 
\end{proof}

The following lemma is a coordinate version of Proposition \ref{prop:abs selfadj opspace}.

\begin{lemma}\label{lem:wlog_sa}
    If $\cE$ is an operator space structure on $\bC^d$ determined by a basis $Q_1, \ldots, Q_d \in B(\cH)$ which is invariant under taking adjoints, then the same operator space structure $\cE$ is determined by a basis consisting of selfadjoint operators 
    \[
    \widetilde{Q}_k = \begin{pmatrix} 0 & Q_k \\ Q_k^* & 0 \end{pmatrix}. 
    \]
\end{lemma}

The following theorem follows from Theorem \ref{thm:selfadj_ineq}, but it is strictly weaker since it requires equality of the spectral radii for all operator tuples, not just all matrix tuples. 
We shall use the following notation. 
Let $\cE_1$ and $\cE_2$ be two operator space structures on $\bC^d$, and we let $E_1 = \cE_1^*$ and $E_2 = \cE_2^*$ denote their duals. 
We let $Q^{(i)}_1, \ldots, Q^{(i)}_d$ denote a basis for $\cE_i$ inducing the operator structure on $\bC^d$. 
We also let $Z^{(i)} = (Z^{(i)}_1, \ldots, Z^{(i)}_d)$ be the coordinate functions on $\bB_{\cE_i}$. 
Then $Z^{(i)}$ is a basis for $E_i$ giving rise to the operator space structure $E_i$ on $\bC^d$. 

\begin{theorem}\label{thm:rho_E1E2}
    Let $\cE_1$ and $\cE_2$ be operator spaces that are invariant under taking adjoints, and put $E_1 = \cE_1^*$ and $E_2 = \cE_2^*$. If $\rho_{\cE_1}(T) = \rho_{\cE_2}(T)$ for every $d$-tuple $T \in B(\cH)^d$ on a Hilbert space $\cH$, then $\cE_1 = \cE_2$. 
\end{theorem}
\begin{proof}
    We will prove that the equality $\rho_{E_1}(T) = \rho_{E_2}(T)$ for every $d$-tuple $T \in B(\cH)^d$ on a Hilbert space $\cH$ implies that $\cE_1 = \cE_2$. The assertion stated in the theorem will follow by duality, since, in this $d$-dimensional setting, $\cE_1 = \cE_2$ if and only if $E_1 = E_2$. 

    By Lemma \ref{lem:sa_equal}, it suffices to prove that $\bB^{sa}_{\cE_1} \subset \bB^{sa}_{\cE_2}$ and that $\bB^{sa}_{\cE_2} \subset \bB^{sa}_{\cE_1}$, and we shall content ourselves with one direction. 
    If the assumption $\rho_{E_1}(T) = \rho_{E_2}(T)$ holds for every $T$ then in particular 
    \be\label{eq:1}
    \rho_{E_1}(Q^{(2)}) = \rho_{E_2}(Q^{(2)}) .
    \ee
    By Lemma \ref{lem:wlog_sa} we can assume that $Q^{(1)}$ and $Q^{(2)}$ are tuples of selfadjoint operators.
    By Lemma \ref{lem:specrad_homo} we know that $\rho_{E_2}(Q^{(2)}) = 1$ and so by \eqref{eq:1} we obtain $\rho_{E_1}(Q^{(2)}) = 1$. By Theorem \ref{thm:rVSdom}, for every $r<1$ there exists $R>r$ such that $R\bB^{(\infty)}_{\cE_1} \subset \dom^{(\infty)}(L_{Q^{(2)}}^{-1})$. 
    Consequently, 
    \be\label{eq:B_subset_dom}
    \bB_{\cE_1} \subset \dom(L_{Q^{(2)}}^{-1}) . 
    \ee
    This means that $I - \sum_k X_k \otimes Q^{(2)}_k$ is invertible for every $X \in \bB_{\cE_1}$. 
    Suppose, contra positively, that $\bB^{sa}_{\cE_1}$ is not contained in $\bB^{sa}_{\cE_2}$. 
    Then there exists $X \in \bB^{sa}_{\cE_1}$ such that $\|X\|_{\cE_2} \geq 1$, and in fact we may assume that
    \[
    \left\|\sum_k X_k \otimes Q^{(2)}_k \right\| = \|X\|_{\cE_2} = 1.
    \]
    But since $\sum X_k \otimes Q^{(2)}_k$ is selfadjoint we have that either $1$ or $-1$ is in the spectrum. 
    Assume without loss of generality that $1$ is in the spectrum. Then $X \notin \dom(L_{Q^{(2)}}^{-1})$, in contradiction with \eqref{eq:B_subset_dom}. 
    This contradiction forces $\bB^{sa}_{\cE_1} \subset \bB^{sa}_{\cE_2}$. 
    As noted at the beginning of the proof, an application of Lemma \ref{lem:sa_equal} now finishes the proof. 
\end{proof}

\begin{corollary}\label{cor:rho_Vmax_Vmin}
    Let $V$ be a conjugation invariant normed space. If $\rho_{\min(V)}(T) = \rho_{\max(V)}(T)$ for every $d$-tuple $T \in B(\cH)^d$ on a Hilbert space $\cH$, then $\min(V) = \max(V)$ and there is a unique operator space on $V$. 
\end{corollary}
\begin{proof}
    By Lemma \ref{lem: isometric adjoint} the operator spaces $\min(V)$ and $\max(V)$ are invariant under taking adjoints, so Theorem \ref{thm:rho_E1E2} applies.
\end{proof}

\section{Consequences of \texorpdfstring{$\rho_{\cE} = \rho_{\cF}$}{rhoE = rhoF} in general}\label{sec:rho_E_F}

In Section \ref{subsec:general} and in the previous section, we showed that selfadjoint operator spaces are distinguished by the corresponding spectral radii. 
Our results leave open the question of what are the consequences of $\rho_{\cE} = \rho_{\cF}$ for a pair of distinct operator spaces $\cE$ and $\cF$. 
In this section we present an NC function-theoretic consequence. 

\begin{remark}\label{rem:row_col}
    In Example \ref{ex:row_col_equal}, we saw that $\rho_{\operatorname{row}}(X) = \rho_{\operatorname{col}}(X)$ for every $X \in \bM^d$, but that these spectral radii functions do not agree on all operator tuples, since $\rho_{\operatorname{row}}(L) \neq \rho_{\operatorname{col}}(L)$ for the NC shift $L$ on the full Fock space. 
    We do not have an example of a pair of distinct operator spaces $\cE$ and $\cF$ such that $\rho_{\cE}(T) = \rho_{\cF}(T)$ for every operator tuple $T$. 
\end{remark}

For an operator space structure $\cE$ on $\C^d$, we let $\cO(\bB_{\cE})$ denote the algebra of all NC functions $f$ defined on $\bB_{\cE}$ such that for every $0 < r < 1$, $f$ is uniformly bounded on $r \overline{\bB_{\cE}}$. In particular, by the results of \cite{KVV}, we know that the Taylor-Taylor expansion of $f$ converges uniformly on every sub-ball of $\bB_{\cE}$. We can define a topology on $\cO(\bB_{\cE})$ by the obvious collection of seminorms
\[
\sigma_r(f) = \sup\{ \|f(X)\| \mid X \in r \overline{\bB_{\cE}} \}, \text{ for } 0 < r < 1.
\]
Moreover, one can similarly define a matrix family of seminorms on $M_n(\cO(\bB_{\cE}))$ for every $n \in \bN$. This topology makes $\cO(\bB_{\cE})$ into a Frechet-Montel locally multiplicatively convex algebra. The matricial structure makes it a local operator algebra in the sense of Effros and Webster \cite{EffWeb}. Let us denote the $d$-tuple of generators of $A(\bB_{\cE})$ by $Z = (Z_1,\ldots,Z_d)$. It is not hard to check that
\[
\cO(\bB_{\cE}) = \varprojlim_{r \to 1^+} A(r \bB_{\cE}) = \varprojlim_{r \to 1^+} H^{\infty}(r\bB_{\cE}).
\]
In particular, $\sigma_r(f) = \|f(rZ)\|$. See \cite{AMS-entires} for the proofs for the case of the row ball, which generalize verbatim to this setting.

\begin{theorem}\label{thm:spr_OBE}
    Let $\cE_1$ and $\cE_2$ be two operator space structures on $\C^d$ and let $E_1 = \cE_1^*$ and $E_2 = \cE_2^*$. If $\rho_{\cE_1}(T) = \rho_{\cE_2}(T)$ for every $T \in B(\cH)^d$, then $\cO(\bB_{\cE_1}) = \cO(\bB_{\cE_2})$ as sets of functions. Moreover, the identity map is a complete homeomorphism. 
\end{theorem}
\begin{proof}
    As above, we will denote by $Z^{(i)}$ the $d$-tuple of generators for $E_i$, for $i = 1,2$. By Lemma \ref{lem:specrad_homo}, we know that $\rho_{\cE_2}(r Z^{(2 )}) = r$, so by assumption $\rho_{\cE_1}(r Z^{(2)}) = r$. In particular, for every $r < s < 1$, $r Z^{(2)}$ is similar to a point in $s \bB^{(\infty)}_{\cE_1}$. Therefore, by Lemma \ref{lem:specrad_strict}, the evaluation map at $r Z^{(2)}$ gives rise to a completely bounded homomorphism on $A(s \bB_{\cE_1})$. In particular, there exists a constant $C_{r,s} > 0$, such that $\|g(r Z^{(2)})\| \leq C_{r,s} \|g\|_{A(s \bB_{\cE_1})}$ for every $g \in A(s \bB_{\cE_1})$. Now let $f \in \cO(\bB_{\cE_1})$. Write $f = \sum_{n=0}^{\infty} f_n$ its Taylor-Taylor expansion in terms of homogeneous components. Fix $s < s' < 1$. By the bound above and the Cauchy estimates in \cite[Theorem 7.21]{KVV} (in particular Equation (7.26) there with $\epsilon = s'/s - 1$), we get that
    \[
    \|f_n(r Z^{(2)})\| \leq C_{r,s} \|f_n\|_{A(s\bB_{\cE_1})} \leq \frac{C_{r,s}}{(1+s'/s - 1)^n} \sigma_{s'}^{(1)}(f) = C_{r,s} (s/s')^n \sigma_{s'}^{(1)}(f).
    \]
    Therefore, the series of $f$ converges absolutely and uniformly on $r \bB_{\cE_2}$ and we get the estimate $\sigma_r^{(2)}(f) \leq \frac{C_{r,s}}{s' - s} \sigma^{(1)}_{s'}(f)$. In other words, $f \in \cO(\bB_{\cE_2})$ and the identity is a continuous map between the two topologies. By symmetry, the identity is a homeomorphism. By the complete boundedness of the evaluation map, the exact same argument applies to matrix-valued functions. Hence, the identity is a complete homeomorphism.
\end{proof}

\begin{remark}
    The above theorem can be viewed through the lens of several variable complex analysis. By \cite{shalsham2025spectral}, for an operator space $\cE$, the function $\log(\rho_{\cE}(\cdot))$ is plurisubharmonic on every level of $\bM^d$. It gives us a plurisubharmonic, albeit not $C^{\infty}$, exhaustion of the similarity envelope of $\bB_{\cE}$. Note that every NC function extends uniquely to the similarity envelope of its original set of definition. Therefore, it is quite natural to say that the exhaustion functions agree if and only if the collection of ``all'' NC analytic functions on the domains agree.
\end{remark}

\begin{example} \label{ex:insufficent_OBE} 
    Consider the row space $\cE_1 = \left(\ell^2_d \right)_{\operatorname{row}}$ and the column space $\cE_2 = \left(\ell^2_d  \right)_{\operatorname{col}}$ we have $\cO(\bB_{\cE_1}) = \cO(\bB_{\cE_2})$, even though $\rho_{\operatorname{row}}(X) = \rho_{\operatorname{col}}(X)$ for all matrix tuples but not for all operator tuples (Remark \ref{rem:row_col}). 

    Let $f \in \C \langle z_1,\ldots,z_d \rangle$ be homogeneous. We note that $\|f\|_{A(\bB_{\operatorname{row}})} = \|f(L)\| = \|f\|_{\cF^2_d}$. Moreover, the transpose map $\cdot^t \colon \cF^2_d \to \cF^2_d$, which is the linear extension of the map sending a word to its reverse, is a unitary. Hence, $\|f(L)\| = \|f^t(L)\|$. Similarly, if we set $f^* = \bar{f}^t$, where $\bar{f}$ is the polynomial obtained by $f$ by applying the complex conjugation to the coefficients, we get $\|f^*(L)\| = \|f(L)\|$. Therefore,
    \begin{multline*}
    \|f\|_{A(\bB_{\operatorname{col}})} = \sup\{ \|f(X)\| \mid X \in \overline{\bB_{\operatorname{col}}} \} = \sup\{ \|f(X)^*\| \mid X \in \overline{\bB_{\operatorname{col}}} \} = \\ \sup\{ \|f^*(X^*)\| \mid X \in \overline{\bB_{\operatorname{col}}} \} = \sup\{f^*(Y)\| \mid Y \in \overline{\bB_{\operatorname{row}}} \} = \|f^*\|_{A(\bB_{\operatorname{row}})} = \|f\|_{A(\bB_{\operatorname{row}})}.
    \end{multline*}
    Now for $g \in \cO(\bB_{\operatorname{row}})$ write $g = \sum_{n = 0}^{\infty} g_n$ its Taylor-Taylor decomposition. Fix $0 < r < s$ and note that for every $n \in \bN$,
    \[
    \|g_n\|_{A(r \bB_{\operatorname{row}})} = \|g_n(r L)\| = \frac{r^n}{s^n} \|g_n(s L)\| = \frac{r^n}{s^n} \left\| \frac{1}{2\pi} \int_0^{2 \pi} e^{- i n \theta} g(s e^{i\theta} L) d\theta \right\| \leq \frac{r^n}{s^n} \sigma_s^{\operatorname{row}}(g).
    \]
    However, by what we have observed $\|g_n\|_{A(r \bB_{\operatorname{row}})} = r^n \|g_n\|_{A(\bB_{\operatorname{row}})} = r^n \|g_n\|_{A(\bB_{\operatorname{col}})} = \|g_n\|_{A(r \bB_{\operatorname{col}})}$. Therefore, the series of $g$ converges on $r \overline{\bB_{\operatorname{col}}}$ absolutely and uniformly. Moreover, we get $\sigma_r^{\operatorname{col}}(g) \leq \frac{s}{s - r} \sigma_s^{\operatorname{row}}(g)$. By symmetry, we get that $\cO(\bB_{\operatorname{row}}) = \cO(\bB_{\operatorname{col}})$ and the identity map is a continuous homeomorphism. The same game can be played with matrix-valued homogeneous polynomials, so the identity is in fact a complete homeomorphism.
\end{example}

\begin{remark}
        One is led to wonder whether it might be possible to obtain $A(\bB_{\cE_1}) = A(\bB_{\cE_2})$ from equality of the spectral norms. We do not know the answer, but we can show that equality of the spectral radii on matrix tuples is not sufficient to imply that. Consider again the row space $\cE_1$ and the column space $\cE_2$. If we consider for $d = 2$ the polynomials $p_n(z_1,z_2) = \sum_{k=1}^n z_1^k$ and $q_n(z_1,z_2) = \frac{1}{\sqrt{n}} p_n(z_1,z_2) z_2$. Then, $q_n(L_1,L_2)^* q_n(L_1,L_2) = \frac{1}{n} L_2^* p_n(L_1,L_2)^* p_n(L_1,L_2) L_2 = 1$. Hence, $\|q_n\|_{A(\bB_{\cE_1})} = 1$. However, we note that for every $(X_1,X_2) \in \overline{\bB_{\cE_2}}$, $(X_1^*,X_2^*) \in \overline{\bB_{\cE_1}}$. Hence,
    \begin{multline*}
    \|q_n\|_{A(\bB_{\cE_2})} = \sup\{\|q_n(X_1,X_2)\| \mid (X_1,X_2) \in \overline{\bB_{\cE_2}}\} = \\ \sup\{ \|q_n^t(X_1^*,X_2^*)\| \mid (X_1,X_2) \in \overline{\bB_{\cE_2}} \} = \|q_n^t\|_{A(\bB_{\cE_1})}
\end{multline*}
    Here $q_n^t(z_1,z_2) = \frac{1}{\sqrt{n}} z_2 p_n^t(z_1,z_2) = \frac{1}{\sqrt{n}} z_2 p_n(z_1,z_2)$. However, since $L_2$ is an isometry, we have that
    \[
    \|q_n^t(L_1,L_2)\| = \frac{1}{\sqrt{n}} \|p_n(L_1,L_2)\| = \sqrt{n}.
    \]
    Therefore, the identity is not a bounded homomorphism. However, for every $0 < r < 1$, we have that
    \begin{multline*}
    \|q_n(rL_1,rL_2)\| = \frac{r}{\sqrt{n}} \sqrt{\sum_{k=1}^n r^{2k}} = \frac{r^2}{\sqrt{n}} \dfrac{\sqrt{1 - r^{2n}}}{\sqrt{1-r^2}} \text{ and } \\ \|q_n^t(rL_1,rL_2)\| = \frac{r}{\sqrt{n}} \|p_n(r L_1, r L_2)\| = \frac{r}{\sqrt{n}} \sum_{k=1}^n r^k = \frac{r^2}{\sqrt{n}} \dfrac{1 - r^{n}}{1 - r}. 
    \end{multline*}
    Therefore, for every $0 < r < s < 1$,
    \[
    \lim_{n \to \infty} \frac{\|q_n^t(rL_1,rL_2)\|}{\|q_n(sL_1,sL_2)\|} = \frac{r^2}{s^{2}} \frac{\sqrt{1 - s^2}}{1 - r}.
    \]
    We conclude that there exists $C_{r,s}$, such that $\|q_n\|_{A(r \bB_{\cE_2})} \leq C_{r,s} \|q_n\|_{A(s \bB_{\cE_1})}$. Therefore, this classical example does not provide a contradiction to the main inequality in the proof of the theorem.
\end{remark}

\subsection*{Acknowledgments} The authors acknowledge the use of GPT-5.2 Pro (OpenAI) as a computational and an exploratory aid and in particular as a tool in the search for the counterexample presented in Example \ref{ex:not_persist}.

\bibliographystyle{abbrv}
\bibliography{spectral_radius_more}

\end{document}